\documentclass[11pt]{article}
\usepackage{makeidx}
\makeindex
\usepackage[english]{babel}

\usepackage{amsmath, amsthm, amssymb}
\usepackage[utf8]{inputenc}

\usepackage{graphicx}
\usepackage{psfrag}
\usepackage{color}
\usepackage{comment}

\newcommand{\R}{\mathbb{R}}

\newcommand{\N}{\mathbb{N}}

\newcommand{\B}{\mathbb{B}}

\def\dd{\mathrm{d}}

\newtheorem{thm}{Theorem}[section]
\newtheorem{cor}[thm]{Corollary}
\newtheorem{lem}[thm]{Lemma}
\newtheorem{prop}[thm]{Proposition}

\newtheorem{defn}[thm]{Definition}
\newtheorem{rem}[thm]{Remark}
\numberwithin{equation}{section}

\begin{document}

\title{Necessary Conditions for Free End-Time Sweeping Process Problems}


\author{ Piernicola Bettiol\footnote{ {\it Univ Brest, UMR CNRS 6205, Laboratoire de Math\'ematiques de Bretagne Atlantique, F-29200 Brest, France, e-mail: \/}
{\tt piernicola.bettiol@univ-brest.fr}}, $\;$  
Paulin Bruneau\footnote{ {\it Univ Brest, UMR CNRS 6205, Laboratoire de Math\'ematiques de Bretagne Atlantique, F-29200 Brest, France, e-mail: \/}
{\tt paulin.bruneau@univ-brest.fr}}
 $\;$ and $\;$
J\'er\'emy Rouot\footnote{ {\it  Univ Brest, UMR CNRS 6205, Laboratoire de Math\'ematiques de Bretagne Atlantique, F-29200 Brest, France, e-mail: \/}
{\tt jeremy.rouot@univ-brest.fr}}}
  
	\maketitle

\begin{abstract}
{\small 
This paper derives necessary optimality conditions for free end-time optimal control problems governed by sweeping process dynamics. In contrast to classical optimal control problems, the presence of the normal cone term in the state equation introduces discontinuities with respect to the state variable and prevents the standard Lipschitz regularity assumptions typically used in optimal control theory. As a result, both the analysis of trajectories and the formulation of optimality conditions require refined tools. 
We address problems with general endpoint constraints and time-dependent data, considering both the case in which the dynamics are Lipschitz continuous in time and the more challenging setting of merely measurable time dependence. Special attention is devoted to the additional transversality conditions associated with the free end-time variable. In the Lipschitz case, these conditions are expressed in terms of a function of bounded variation that coincides almost everywhere with the maximized Hamiltonian. In the measurable case, we show that a suitable interpretation can be obtained via the notions of sub- and super-essential values of the maximized Hamiltonian, recently introduced in the literature, leading to strengthened transversality conditions. These conditions also involve an additional term reflecting the interaction between the costate arc and the normal cone to the moving set at the optimal endpoints.
Our approach combines suitable perturbation techniques with a penalization method tailored to sweeping processes. The moving set is modeled as a finite intersection of inequality-defined sets satisfying suitable constraint qualifications, including a positive linear independence condition and a diagonal dominance property of the associated Gramian matrix. We also introduce a local version of the constraint qualification condition. Our results apply to a broad class of problems, including those involving unbounded moving sets such as polyhedra. The resulting optimality system extends and sharpens recent results for classical free end-time problems to the non-Lipschitz, discontinuous framework of sweeping processes. The relevance of the transversality conditions obtained in this article is illustrated through two simple examples.
}
\vspace{0.05 in}

\noindent
{\small Keywords: Free End-Time Problems, Sweeping Processes, Necessary Conditions, Nonsmooth Analysis}

\bigskip
\noindent
{\small AMS subject classifications: 49K15, 49K21, 49J52.}

\end{abstract}

\section{Introduction}

In this paper, we derive first order necessary conditions of optimality for free end-time dynamic optimization problems involving sweeping processes of the following form:

\[ (P) \hspace{0.4 in} \left\{
\begin{array} {l}
\mbox{Minimize } g(S,x(S),T,x(T)) \\
\mbox{over intervals } [S,T], \mbox{ arcs }
x  \in W^{1,1}([S,T];\R^n) \\
\mbox{\hspace{0.5 in}and measurable functions }
u :[S,T]\rightarrow \R^{m} \mbox{ satisfying } \\
\dot{x}(t) \in f(t,x(t), u(t)) -N_{C(t)}(x(t)) \mbox{ a.e. } t \in [S,T], \\
u(t) \in U(t) \mbox{ a.e. } t \in [S,T],
\\
(S, x(S),T, x(T))\in E\,.
\end{array}
\right.
\]

\noindent The data for this problem comprise positive integers $n$ and $m$,
functions $g:\R \times \mathbb{R}^{n}\times \R \times \mathbb{R}^{n} \to \R$ and
$f:\R \times \R^n\times \R^{m} \rightarrow \R^n$, non-empty multifunctions $U:\R \leadsto \R^{m}$ and $C:\R \leadsto \R^n$, and a closed set $E\subset \R^{1+n+1+n}$. 

Problem (P) is referred to as a `\textit{free end-time}' problem because the end-times $S$ and $T$ for the arcs $x  \in W^{1,1}([S,T];\R^n)$ (the \textit{state trajectories}) and measurable functions $u :[S,T]\rightarrow \R^{m} $ (the \textit{controls}) in (P) are considered as part of the choice variables (and they also appear in the cost $g$ to be minimized and in the endpoint constraint $E$). 
Then, a {\em process} for the sweeping process of (P) is taken to be a triple $([S,T],x ,u)$ in which $[S,T]$ is a
\textit{time interval}, $u:[S,T] \rightarrow \R^{m}$ is a control and the state trajectory $x \in W^{1,1}([S,T];\R^{n})$ is the corresponding solution of the differential equation satisfying
\begin{equation*}
    (SP)\; \left\{
    \begin{aligned}
        &\dot{x}(t) \in  f(t,x(t),u(t)) -N_{C(t)}(x(t)), \qquad \text{a.e. } t \in [S,T],\\
        &u(t) \in U(t), \qquad \text{a.e. } t \in [S,T],\\
        &x(S)\in C(S).
    \end{aligned}
    \right.
\end{equation*}
An important special case of problem (P) is the minimum time optimal control problem, in which the trajectory starts from a prescribed initial state  {$x_0\in C(S_0)$} at a (fixed) initial time $S_0$ and is required to reach a given target $E_T\subset \R^n$ (in this case, we take $E=\{S_0\} \times \{x_0\} \times \R \times E_T$). 

We say that a process $([S,T],x,u)$ is admissible for $(P)$ if it satisfies all the constraints of $(P)$, that is it is a solution of $(SP)$ and the end-point constraint $(S,x(S),T,x(T)) \in E$ is satisfied.

An admissible process $([\bar{S},\bar{T}], \bar{x} ,\bar{u})$, is said to be a local minimizer for $(P)$
if there exists $\delta > 0$ such that
$$
g(S,x(S),T,x(T)) \geq g(\bar{S},\bar{x}(\bar{S}),\bar{T},\bar{x}(\bar{T}))
$$
for every admissible process $([S,T],x,u)$ of ($P$) such that
\begin{equation}\label{metric}
|S - \bar S  |    +  |T - \bar T | + ||x_e - \bar x_e ||_{L^\infty([S\wedge \bar S,T \vee \bar T];\R^n)} \leq  \delta.
\end{equation}
Here, given $a,b \in \R$, we set $a\wedge b:= \min\{a,b\}$, $a\vee b:= \max\{a,b\}$ and $x_e$ denotes the extension to $(-\infty, + \infty)$ of the function $x:[S,T] \rightarrow \R^n$, defined as follows: 
$$
x_e(t):= \left\{ 
\begin{array}{ll}
x(S)& \mbox{if  } t < S
\\
x(t)& \mbox{if  } t \in [S,T] 
\\
x(T)& \mbox{if  } t > T\,.
\end{array}
\right.
$$ 
According to this extension, we set $\dot{x}_e(t):=0$ whenever $t \notin [S,T]$.

It is well known in the literature (cf. \cite{CV89_1,CV89_2,Vinter2010,Bettiol-Vinter-book2024} and the references therein) that, when deriving necessary optimality conditions for `classical' (i.e. in the absence of the sweeping process term) free end-time problems, it is also crucial to take into account the additional information arising from the extra degree of freedom provided by the end-time variable. When $U(.)\equiv U$ is constant and the dynamics function $f$ is Lipschitz continuous w.r.t. the time variable $t$ (the autonomous case being a particular instance of this class), the standard interpretation of these additional necessary conditions is typically expressed in terms of the boundary conditions of an absolutely continuous function $a(.)$ evaluated at the optimal end-times $\bar S$ and $ \bar T$. This function coincides almost everywhere with the maximized Hamiltonian function evaluated along the minimizer $\bar x$ and the corresponding costate arc $p$, namely the mapping $t\mapsto H(t,\bar x(t), p(t))$.

However, when the data are merely time-measurable, this is no longer possible in general and therefore requires a suitable interpretation. A well-established answer to this issue, due to Clarke and Vinter \cite{CV89_1,CV89_2}, is to express the boundary conditions arising from the extra degree of freedom in terms of the essential values of the measurable function $t \mapsto H(t, \bar x (t), p(t))$ evaluated at the optimal end-times $\bar S$ and $ \bar T$. In a recent paper \cite{Bettiol-Vinter-SICON2025}, the new notions of sub- and super-essential values have been introduced. These concepts allow improvements of earlier results and yield stronger transversality conditions for free end-time sweeping process problems. The sub- and super-essential values associated with the maximized Hamiltonian $t \mapsto H(t, \bar x (t), p(t))$ correspond, respectively, to the limiting subdifferential and superdifferential of the associated integral function (cf. Definition \ref{def1} below).

Problem (P) differs substantially from classical free end-time dynamic optimization problems due to the presence of the term `$N_{C(t)}(x)$', which is discontinuous with respect to $x$ on the boundary of $C(t)$, on the right side of the dynamic constraint in (P). As a consequence, the standard Lipschitz regularity with respect to the state variable (usually required in the optimal control literature to derive first-order necessary conditions) is no longer guaranteed. Given a control function $u$, establishing the existence of a solution $x$ to the sweeping dynamic constraint (SP) may even become problematic unless suitable assumptions on the moving set 
$C(t)$ are imposed (with \textit{prox-regularity} being a standard one; see, for example, \cite{ANTh,N_Th,CT,BeVe,Thib} along with the references they cite). Therefore, dynamic optimization problems involving sweeping processes are somewhat nonstandard, since the dynamic constraint exhibits a discontinuous behaviour on the boundary of $C(t)$, and thus represents a class of highly nontrivial problems. 
Only in the last decade there has been significant progress in the optimal control of sweeping processes, with several sets of necessary optimality conditions being derived. Starting with \cite{BK}, various approaches for deriving such conditions are now available in the literature and may be summarized as follows. These include an exponential penalization approach coupled with a suitable constraint qualification, as well as a penalization method involving the distance function that provides more refined convergence estimates for controls and trajectories (cf., respectively, \cite{dPFS2019,dPFS2023,NZ2023,NZ2024,ChamZeid} and \cite{arc,BCG2025}, and cited works). Alternative techniques rely on reformulating the sweeping process problem as a classical state-constrained problem; however, the resulting maximum principle is not necessarily in normal form and may exhibit degeneracy (cf. \cite{HP2022} and, for discrete approximation or catching-up techniques, \cite{CCMN,GCPG,CMNN2024,CMNNT2025} and references cited therein).

To the best of our knowledge, the only contributions in the literature on necessary optimality conditions for free end-time dynamic optimization problems involving sweeping processes are \cite{CMNN2024} and \cite{CMNNT2025}, and for a particular class of problems. More precisely, in \cite{CMNN2024} an autonomous problem (i.e. $f$ and $U$ are time independent) is investigated; the initial state is fixed, the moving set $C(t)$ is assumed to be a convex polyhedron, and the right endpoint constraint has a separate structure (that corresponds to the particular case of the reference problem (P) obtained by taking $E=\{S_0\} \times \{x_0\} \times E_1 \times E_T$). In \cite{CMNNT2025}, the authors investigate a class of problems that share a structural resemblance with those studied in \cite{CMNN2024}, with the key distinction that the moving set is allowed to depend not only on the state variables but also on the control parameters. This additional dependence introduces new analytical challenges and enriches the problem framework. The authors derive necessary optimality conditions for the associated discrete-time problem and further demonstrate the applicability of their theoretical results through illustrative examples drawn from motion models.

In the present paper, we derive necessary optimality conditions for problem (P), which involves a general left/right endpoint constraint 
$E$ and time-dependent data. We consider both the case in which the dynamics function $f$ is Lipschitz continuous in $t$ (which covers the simpler autonomous setting studied in \cite{CMNN2024} and \cite{CMNNT2025}) and the more complicated time-measurable case, obtaining sets of necessary conditions that extend to problems involving sweeping processes the recent results derived for `classical' free end-time dynamic optimization problems in \cite{Bettiol-Vinter-SICON2025} (cf.\ \cite[Chapter~9]{Bettiol-Vinter-book2024}). We emphasize that, for measurable time-dependent problems, the transversality conditions associated with the additional degree of freedom involve the refined notions of sub- and super-essential values, together with an additional term involving the scalar product between the costate arc $p$ and a normal vector to the moving set $C(t)$ evaluated at the optimal end-times $\bar S$ and $ \bar T$. This aspect is particularly significant, as it reveals the subtle interplay between the transversality conditions and the moving set $C(t)$  that characterizes the problem (see Examples 1 and 2 below and the related comments). In the Lipschitz time-dependent case, we end up again with information provided by the extra degree of freedom (associated with the end-time variable), which is expressed in terms of a function $a(\cdot)$ of bounded variation (in our case), that coincides almost everywhere with the maximized Hamiltonian.

 We shall consider the case in which the time dependent (moving) set $C(t)$ is nonempty and is a finite intersection of sets $C_j(t)$, each admitting a functional inequality representation:
\[
    C(t):=\bigcap_{j=1}^J C_j(t),\qquad
    C_j(t):=\left\{x\in\mathbb{R}^n:\; h_j(t,x)\le 0\right\},
\]
for some given functions $h_j: \R\times \R^n \to \R$,  \(j=1,\dots,J\). As in \cite{BCG2025} and other papers dealing with possibly nonsmooth sets $C(t)$ such as \cite{CMNN2024,dPFS2023,ChamZeid,NZ2024}, we assume that $C(t)$ satisfies a positive linear independence constraint qualification together with a `diagonal dominance condition' of the Gramian matrix of the gradients of the functions $h_j$ near the boundary of the moving set $C(t)$ (and near the reference minimizer). In contrast with \cite{BCG2025}, we work directly with the defining functions $h_j$'s, without passing to the (signed) distances to the sets $C_j(t)$. This simplifies the analysis and thereby avoids the introduction of an additional (very restrictive) technical assumption required in \cite{BCG2025} (see condition (2.9) in \cite{BCG2025}). 
We consider a refined global constraint qualification, and we introduce a `local' constraint qualification. Both conditions allow us to cover situations in which the moving set $C(t)$ is possibly unbounded, such as polyhedra. 
Our local constraint qualification combines a positive linear independence property with a merely local prox-regularity condition, which is required to be satisfied only on a tubular neighbourhood of a reference arc. 
Our proof techniques are based on Clarke's perturbation methods. We point out an additional important improvement: in our approach to obtaining non-degenerate necessary conditions (with general left/right endpoint constraints), one must apply Ekeland's variational principle in an appropriate (non-trivial) way that differs substantially from previous work (see, e.g., \cite{BCG2025}). 
We highlight that our approach allows us to cover general (merely closed) endpoint constraints, whereas in \cite{CMNN2024,dPFS2023,ChamZeid,NZ2024} at least one component of the endpoint constraints must be contained in the graph of $C(.)$. 
A further relevant consequence of our results is that, when the minimizing arc remains in the interior of the moving set (which is always the case when $C \equiv \mathbb{R}^n$), our set of necessary conditions does coincide with the standard necessary conditions (including the transversality condition) that have been established for `classical' (i.e., in the absence of the sweeping process term) free end-time optimal control problems, cf.\ \cite{Bettiol-Vinter-SICON2025} and \cite[Chapter~9]{Bettiol-Vinter-book2024}.
We emphasize also that our results do not require any convexity assumptions on the controlled dynamics function $f$.

\medskip

\noindent
We conclude this introduction by providing points of notation and terminology and by defining the concepts of sub- and super-essential values, which play a crucial role in the transversality conditions in the measurable time-dependent case. Then, the paper is organized as follows. In Section 2, we present the main necessary conditions results for problems involving, first, a measurable time-dependent dynamics function and, subsequently, the Lipschitz time-dependent case. 
The relevance of the new transversality conditions for the measurable time-dependent case is illustrated through two simple examples. In both measurable and Lipschitz time-dependent situations, we provide normality results. Section 3 presents the preliminary analysis, focusing in particular on suitable approximation techniques that form the foundation for constructing the auxiliary problems used in the perturbation approach adopted in our paper. Building on these results, Section 4 provides a proof of the necessary optimality conditions in the measurable time-dependent setting. Finally, Section 5 is devoted to establishing the maximum principle in the case of Lipschitz continuous time dependence.

\ \\

\noindent
{\bf Notation:} Given an interval $[S,T]$ and $p\ge 1$, $L^{p}([S,T];\mathbb{R}^n)$ denotes the space of $\mathbb{R}^n$-valued functions on $[S,T]$ whose $p$-th power is integrable when $1\le p<\infty$, and that are essentially bounded when $p=\infty$. We simply write $L^{p}(S,T)$ for $L^{p}([S,T];\R)$. $W^{1,p}([S,T]; \R^n)$ is the space of absolutely continuous functions whose weak derivative belongs to  $L^{p}([S,T];\mathbb{R}^n)$. 
By $NBV([S,T];\R^{n})$ we denote the space of $\R^n$-valued functions of bounded variation defined on $[S,T]$ which are `{\it normalized}' in the sense that they are right-continuous on $(S,T)$. 
The Euclidean norm of a vector $x\in \R^n$ is written $|x|$. The closed unit ball in Euclidean space is written $\B$. 
${\rm co}\,D$ denotes the convex hull of the set $D$. 
For measurable functions $u:[S,T]\to\R^m$ and $v:[S',T']\to\R^m$, we define
\[
d_{\mathcal E}(u,v) := \mathrm{meas}\,\{\,t\in[S,T] \cap [S',T'] : u(t)\neq v(t)\,\}.
\]
Here, `$\mathrm{meas}$' denotes the Lebesgue measure on $\R$.
For a multifunction $\Gamma:\R \leadsto \R^k$, its graph is  
${\rm Gr}\,\Gamma = \{(t,u)\in \R\times\R^k : u\in \Gamma(t)\}$.
The positive part of a function $h: \R^k \to \R$ is denoted  {by} $h^+(\cdot)=h(\cdot)\vee 0$.
   
We make use of some constructs from nonsmooth analysis, described in detail, for example, in \cite{Vinter2010,Bettiol-Vinter-book2024}. Given a closed set $C\subset \mathbb{R}^n$ and a point $x\in \mathbb{R}^n$, we denote by $d_C(x)$ the distance of $x$ from $C$ and we denote by $\pi_C(x)$ the set of all points in $C$ minimizing this distance. 
 {The boundary of $C$ is written $\partial C$. 
}Given a point $x\in C$,  the {\em proximal normal cone of $C$ at $x$} is defined by

$$
N_C^P(x)\hspace{-1mm}:=\hspace{-1mm} \{
\zeta \in \R^{n}: 
\exists 
M>0 \mbox{ {\it s.t.} } \zeta \cdot (y-x) \leq M|x-y|^{2} \mbox{ {\it for all} } y \in C
\}.
$$

\noindent
The {\em (limiting) normal cone of $C$ at $x$} is
$$
N_C(x):=\{\lim_{i\to\infty}\zeta_i \,:\, \zeta_i\in N^P_C(x_i) \mbox{ {\it and} }x_i\in C \mbox{ {\it for all }}i, \mbox{{\it and} } x_i\to x\}\,.
$$
Given a lower semicontinuous function $f:\R^n\to \R\cup\{+\infty\}$ and a point $ x$ such that $f(x)  < +\infty$, the {\em proximal subdifferential of $f$ at $x$} is 
$$
\partial^Pf(x):= \{
\zeta\in\R^n: \exists \;M>0, \epsilon>0 \mbox{ such that }
  f(y)-f(x)\geq \zeta \cdot (y-x) - M |y-x|^2, \mbox{ for all } y\in x+\epsilon \B \}.
$$
The {\em limiting subdifferential of $f$ at $x$}, $\partial f(x)$, is the set
$$
\partial f(x):=\{\lim_{i\to\infty}\zeta_i:\zeta_i\in \partial^P f(x_i), x_i\to x, f(x_i)\to f(x)\}\,.
$$
The {\em limiting superdifferential of $f$ at $x$}, is the set $-\partial( - f)(x)$.
When  the function $f$ depends on two variables, we write $\partial_x f (\bar x, \bar y)$ to denote the limiting subdifferential of $x\to f(x,\bar y)$ at $\bar x$.  

\noindent
Let $r>0$ be given. A nonempty closed set $C \subset \mathbb{R}^n$ is said to be
$r$-prox-regular if for every $x \in C$, every $\xi \in N_C^P(x)\cap \mathbb{B}$,
and every $t \in (0,r)$, one has $x \in \pi_C(x+t\xi)$.
The set $C$ is said to be $r$-prox-regular at a point $x \in C$ if there exists
a neighborhood $X$ of $x$ such that $C \cap X$ is $r$-prox-regular.

\noindent
 {When we consider the measurable time-dependent case, we employ the notions of sub- and super-essential values of indefinite integral functions introduced in \cite{Bettiol-Vinter-SICON2025} (cf. \cite[Chapter 9]{Bettiol-Vinter-book2024}), which we recall below. 
Take a function $\phi \in  {L^1(S,T)}$.  Fix any $S'\in [S,T]$ and define the absolutely continuous function $\psi:[S,T] \rightarrow \R$ to be 
\begin{equation}
\label{int}
\psi(t):= \int_{S'}^t \phi(s)ds\,.
\end{equation}
}
 
\begin{defn}\label{def1}
The sub- and super-essential values of the function $\phi \in L^1(S,T)$ at $ \bar t \in (S,T)$ are defined as follows.
\begin{itemize}
\item[$(1)$:] {\it (Sub-Essential Value)}
\begin{eqnarray*}
&&
\underset{t \rightarrow \bar t}{\mbox{sub-ess}}\, \phi(t) \,:=  \Big\{ \zeta \in \R \,: \exists \;t_i \rightarrow \bar t \mbox{ and }  \zeta_i \rightarrow \zeta \mbox{ s.t. }
\\
&&
\hspace{1.3 in}
 \underset{\epsilon \downarrow 0}{\limsup} \, \epsilon^{-1} \int_{t_i - \epsilon}^{t_i} \phi(s)ds \leq    \zeta_i   \leq   \underset{\epsilon \downarrow 0}{\liminf} \, \epsilon^{-1} \int_{t_i}^{t_i + \epsilon} \phi(s)ds, \mbox{ for each }i
 \Big\}\,;
\end{eqnarray*}
\item[$(2)$:] {\it (Super-Essential Value)}
\begin{eqnarray*}
&&\underset{t \rightarrow \bar t}{\mbox{super-ess}}\, \phi(t) \,:= \Big\{ \zeta \in \R \,: \exists \; t_i \rightarrow \bar t \mbox{ and }  \zeta_i \rightarrow \zeta \mbox{ s.t. }
\\
&&
\hspace{1.3 in} \underset{\epsilon \downarrow 0}{\limsup} \, \epsilon^{-1}\int_{t_i }^{t_i+ \epsilon} \phi(s)ds \leq    \zeta_i   \leq   \underset{\epsilon \downarrow 0}{\liminf} \epsilon^{-1}\int_{t_i - \epsilon}^{t_i } \phi(s)ds, \mbox{ for each }i
 \Big\}\,.
\end{eqnarray*}
\end{itemize}
Given a function of two variables $\phi(t,x)$,  $\underset{t \rightarrow \bar t}{\mbox{sub-ess}}\, \phi(t,x) $ and  $\underset{t \rightarrow \bar t}{\mbox{super-ess}}\, \phi(t,x) $ indicate the sub- and super-essential values of $t \rightarrow \phi(t,x)$ respectively, for fixed $x$.
\end{defn}
The sub- and super-essential values of $\phi$ satisfy several useful properties, and for these we refer the reader to \cite{Bettiol-Vinter-SICON2025} and \cite[Chapter 9]{Bettiol-Vinter-book2024}. Here, we highlight only that these concepts precisely capture the sub- and super-differentials of the associated indefinite integral function $\psi$ defined in \eqref{int}:
$$
\underset{t \rightarrow \bar t}{\mbox{sub-ess}}\, \phi(t) = \partial \psi(\bar t) \; \mbox{ and } \;  
\underset{t \rightarrow \bar t}{\mbox{super-ess}}\, \phi(t) = - \partial (-\psi)(\bar t)\,.
$$

\section{Necessary Conditions For Free End-Time Sweeping Process Problems}

 {We begin by presenting a maximum principle for the more general case involving measurable time dependent data. We then provide the necessary optimality conditions for the Lipschitz time dependent case.
}

\subsection{The Measurable Time-dependent Case}

In what follows, $I(t,x)$ is the set of {\it active constraints} defined, for all $(t,x) \in \R \times \R^n$, by
\[
I(t,x):=\left\{\,j\in\{1,\dots,J\}:\, h_j(t,x)=0\,\right\}
\]
and, for $\eta>0$, $I_\eta^+(t,x)$ is the set of active indices `near the boundary from the outside':
\[
I_\eta^+(t,x):=\left\{\,j\in\{1,\dots,J\}:\, 0\le h_j(t,x) \le \eta\,\right\}.
\]
 {We shall assume the following hypotheses (H1)--(H5) on the problem data, for a local minimizer $([\bar S, \bar T],\bar x, \bar u)$ and for some $\bar \delta>0$.
}\begin{itemize}
\item[(H1):] The function \(g\) is Lipschitz continuous on $(\bar S, \bar x (\bar S), \bar T,\bar x (\bar T)) + \bar \delta \B$, and \(E\) is a closed subset of $\R\times\R^n\times\R\times\R^n $.

\item[(H2):] \(\mathrm{Gr}\,U\) is an \(\mathcal{L}\times\mathcal{B}^{m}\)-measurable subset of \(\mathbb{R}\times\mathbb{R}^{m}\).

\item[(H3):]
\begin{itemize}
  \item[(i):] For every \(x\in\mathbb{R}^n\), the map \((t,u)\mapsto f(t,x,u)\) is \(\mathcal{L}\times\mathcal{B}^{m}\)-measurable;
  \item[(ii):] there exist $k_f \in L^1(\bar S- \bar \delta, \bar T +\bar \delta)$ and $c_f>0$ such that
  \[
  |f(t,x,u)-f(t,x',u)|\le k_f(t)|x-x'|,\qquad |f(t,x,u)|\le c_f,
  \]
  for all $x,x'\in \bar x_e (t) + \bar \delta \B$ and all $u\in U(t)$, a.e. $t\in [\bar S- \bar \delta, \bar T +\bar \delta] $.
\end{itemize}

 {\item[(H4):] For some $\eta_0>0$ we have: 
\begin{itemize}
  \item[(i):] the time dependent set $C(t)$ is nonempty and has the following representation
    \[
    C(t):=\bigcap_{j=1}^J C_j(t),\qquad
    C_j(t):=\left\{x\in\mathbb{R}^n:\; h_j(t,x)\le 0\right\},
    \]
    where, for each \(j=1,\dots,J\), $h_j: [\bar S- \bar \delta, \bar T +\bar \delta]\times \R^n \to \R$ is of class $C^{1+}$ (i.e. everywhere Fr\'echet differentiable with locally Lipschitz continuous derivatives);
  \item[(ii):] 
  $\nabla_x h_j(t,x)\neq 0$ for all $j \in I(t,x)$, $(t,x) \in [\bar S- \bar \delta, \bar T +\bar \delta] \times \R^n$; there exists $K_h>0$ such that 
  \[
  \sup_{j\in I_{\eta_0}^+(t,x)}| \nabla h_j(t,x)| \le K_h, \quad \text{for all } \; (t,x)\in [\bar S- \bar \delta, \bar T +\bar \delta] \times \R^n;
  \]
  and 
  $D^2_x h_j(t,x)$ exists and is continuous with respect to $x$ on the `tube' $\{(t,x) \in [\bar S- \bar \delta, \bar T +\bar \delta]\times \R^n \,:\, |x-\bar x_e(t)|\leq \bar \delta\}$;
  \item[(iii):] the set $C(t)$ satisfies the following strict positive linear independence constraint qualification: there exists 
  $\gamma_0>0$ such that 
  \[
  \Big|\sum_{j\in I^+_{\eta_0} (t,x)}\alpha_j\nabla_x h_j(t,x)\Big|\ge \gamma_0,  \quad \text{whenever} \sum_{j\in I^+_{\eta_0}(t,x)} \alpha_j=1, \;\; \alpha_j\ge 0
  \]
  for all $x\in \R^n \setminus C(t)$, $t \in [\bar S- \bar \delta, \bar T +\bar \delta]$.
\end{itemize}}

\item[(H5):] For all $R>0$ there exists \(\rho\in(0,1)\)  
such that for each  {\(k\in I(t,x)\)}, we have
\[
\sum_{\substack{ {j\in I(t,x)}\\ j\neq k}}
\bigl|\nabla_x h_k(t,x)\cdot\nabla_x h_j(t,x)\bigr|
\le \rho\,\bigl|\nabla_x h_k(t,x)\bigr|^2,
\]
for all $(t,x) \in [\bar S- \bar \delta, \bar T +\bar \delta] \times R\B$;
\end{itemize}

\noindent
Let ${\cal H}$ denote  the un-maximized
Hamiltonian function for ($P$):
$$
{\cal H}(t,x,p,u):= p\cdot f(t,x,u).
$$
Then, the maximized Hamiltonian $H$ has the representation $H(t,x,p)= \underset{u \in U(t)}{\sup} {\cal H}(t,x,p,u)$.
\begin{thm}\label{pm}
\textbf{(Free End-Time Maximum Principle - Measurable Time Dependence)}
Let $([\bar S,\bar T],\bar x,\bar u)$ be a local minimizer for $(P)$ such that $\bar T - \bar S>0$.
Assume that {\rm (H1)}--{\rm (H5)} hold for some $\bar \delta>0$.
Then, there exist a function $p\in NBV([\bar S,\bar T];\R^n)$, a scalar $\lambda\ge 0$, finite signed Radon measures $\dd\mu_j$ on $[\bar S,\bar T]$, and bounded measurable functions $\eta_j:[\bar S,\bar T]\to [0,+\infty)$ for $j=1,\dots,J$, such that 
\begin{itemize}
    \item[(i):] $(p,\lambda)\neq 0$;
    
    \item[(ii):] $\dot{\bar x}(t)= f(t,\bar x(t),\bar u(t))- \sum_{j=1}^J \eta_j(t)\,\nabla_x h_j(t,\bar x(t)), \qquad \eta_j(t)\neq0 \Rightarrow h_j(t,\bar x(t))=0, \text{ a.e. } t\in[\bar S,\bar T] $;
    
    \item[(iii):] $-\dd p \in \operatorname{co}\,\partial_x\mathcal H(t,\bar x(t),p(t),\bar u(t))\,\dd t -\sum_{j=1}^{J}\nabla_x h_j(t,\bar x(t))\,\dd\mu_j(t) - p(t)\sum_{j=1}^{J}\eta_j(t)\,D^2_x h_j(t,\bar x(t))\,\dd t$;
    
    \item[(iv):] there exist $\nu_0 \in N_{C(\bar S)}(\bar x (\bar S))$, $\nu_1 \in N_{C(\bar T)}(\bar x (\bar T))$,
    \[
    \begin{aligned}
    \xi_0 &\in 
		\underset{t\to \bar S}{sub\text{-}ess}\,
		H(t,\bar x(\bar S),p(\bar S))
    \ - \ p(\bar S)\cdot \nu_0, \\
    \xi_1 &\in 
		\underset{t\to \bar T}{super\text{-}ess}\,
		H(t,\bar x(\bar T),p(\bar T))
      \ - \ p(\bar T)\cdot \nu_1,
    \end{aligned}
    \]
    such that
    $
    (-\xi_0,p(\bar S),\xi_1,-p(\bar T))
    \in
    \lambda\,\partial g(\bar S,\bar x(\bar S),\bar T,\bar x(\bar T))
    + N_E(\bar S,\bar x(\bar S),\bar T,\bar x(\bar T))$
    
    \hspace{3.2 in} $ + \sum_{j\in I(\bar S,\bar x(\bar S))} \beta_j\,\nabla h_j(\bar S,\bar x(\bar S))\times\{0\}\times\{0\}$
    
		for some $\beta_j\ge 0$,  $j\in I(\bar S,\bar x(\bar S))$;
    \item[(v):] $p(t)\cdot f(t,\bar x(t),\bar u(t))=\max_{u\in U(t)} p(t)\cdot f(t,\bar x(t),u)$, \quad a.e.\ $t\in[\bar S,\bar T]$;

    \item[(vi):] $\eta_j(t)\,p(t)\cdot\nabla_x h_j(t,\bar x(t))=0$,
    \quad for each $j\in I(t,\bar x(t))$, a.e.\ $t\in[\bar S,\bar T]$;
    
    \item[(vii):] $\operatorname{supp}(\dd\mu_j)\cup \operatorname{supp}(\eta_j) \subset \{\,t\in[\bar S,\bar T]: h_j(t,\bar x(t))=0\,\}$, \quad for each $j=1,\dots,J$;
    
    \item[(viii):] $p(t)\cdot\nabla_x h_j(t,\bar x(t))\,\dd\mu_j(t)$
    \text{is a nonnegative measure}, \quad for each $j=1,\dots,J$.
\end{itemize}
\end{thm}

\ \\

\begin{rem}\label{nc}
\begin{itemize}
    \item[(a)]  The strict positive linear independence condition {\rm (H4)} (iii) is equivalent to the positive linear independence condition when the moving set $C(t)$ is uniformly bounded; this is a constraint qualification widely employed in deriving necessary conditions for sweeping process problems (cf. \cite{BCG2025,CMNN2024,ChamZeid} for detailed discussion and further references). Here, we recall only some (simple) consequences that follow from {\rm (H4)} (iii): for all $t\in[\bar S - \bar \delta,\bar T + \bar \delta]$, $C(t)$ has nonempty interior (i.e. ${\rm  int }\,  C(t)\neq \emptyset$ for all $t$), 
		and 
    \[
    N_{C(t)}(x) = \left\{\xi \in \R^n : \xi = \sum_{j \in I(t,x)} \alpha_j \nabla_x h_j(t,x), \quad \alpha_j\ge 0 \right\}
    \]
    for all $x \in \partial C(t) \cap (\bar x_e(t) + \bar \delta \B)$.
    Observe that the boundary of $C(t)$ satisfies 
    \[
    \partial C(t) = \{x \in C(t) : I(t,x) \neq \emptyset\}.
    \]
		We also observe that for the inequality in {\rm (H4)}(iii), it suffices to consider points $x$ not belonging to $C(t)$. This is because the inequality only plays a role in estimates involving points of trajectories that lie outside the moving set $C(t)$ (see the analysis of the proofs in Section 4).
    \item[(b)] Notice that, when the reference minimizer $([\bar S,\bar T],\bar x,\bar u)$  is such that $\bar x(t) \in {\rm  int }\,  C(t)$ for all $t\in [\bar S,\bar T]$ (for instance, when $C(t)=\R^n$ for all $t$), the necessary optimality conditions of Theorem~\ref{pm} are consistent with those obtained in \cite{Bettiol-Vinter-book2024,Bettiol-Vinter-SICON2025} for free end-time problems. In particular, since in Theorem~\ref{pm} the endpoint constraint $E$ is allowed to be a general closed set, the transversality condition coincides exactly with what one would expect from \cite{Bettiol-Vinter-book2024,Bettiol-Vinter-SICON2025}.
    \item[(c)]  {The transversality condition (iv) of Theorem \ref{pm} differs from that one established for optimal control problems without the sweeping process term `$N_{C(t)}$' in the dynamics provided by \cite{Bettiol-Vinter-book2024,Bettiol-Vinter-SICON2025}. This is due to the presence of the additional terms `$p(\bar S)\cdot \nu_0$' and `$p(\bar T)\cdot \nu_1$'. }These can not be substituted by terms involving the sub- or super-essential values of the Hamiltonian. Indeed, we can not even use a `less refined' set including the essential value and replace the terms $p(\bar S)\cdot\nu_{0}$ and 
    $p(\bar T)\cdot\nu_{1}$ in (iv) of Theorem \ref{pm} by essential value sets such as 
    \[
        \operatorname*{ess\,}_{t \to \bar S} 
        \Big[\, p(\bar S)\cdot \sum_{j=1}^{J}\eta_j(t)\,D^2_x h_j(\bar S,\bar x(\bar S)) \Big]
        \quad\text{and}\quad
        \operatorname*{ess\,}_{t \to \bar T} 
        \Big[\, p(\bar T)\cdot \sum_{j=1}^{J}\eta_j(t)\,D^2_x h_j(\bar T,\bar x(\bar T)) \Big].
   \]
   The essential values of an integrable function $\phi (t)$ at $t=\bar t$ being defined as the set 
   \[
   \underset{t \rightarrow \bar t}{\mbox{ess}}\, \phi (t) \,:= \, \Big[ \lim_{\epsilon \downarrow 0}\,
    \underset{\bar t- \epsilon \leq  {t \leq \bar t+\epsilon}}{\mbox{ess inf}}\phi (t)\,,\,
    \lim_{\varepsilon \downarrow 0}\,
    \underset{\bar t- \epsilon \leq { t \leq \bar t+\epsilon}}{\mbox{ess\,sup}}  \phi (t)\Big]\,.
   \]
   (Recall that the essential value of $\phi$ contains the sub- and super-essential values, cf. \cite{Bettiol-Vinter-book2024,Bettiol-Vinter-SICON2025}). This is illustrated in Example 2 below.  It is immediate to see that the additional term `$p(\bar S)\cdot \nu_0$' (respectively, `$p(\bar T)\cdot \nu_1$') vanishes when $\bar x(\bar S) \in {\rm  int }\,  C(\bar S)$ (respectively, $\bar x(\bar T) \in {\rm  int }\,  C(\bar T)$). Observe also that assumptions (H2) and (H3) guarantee that the sub- and super-essential values of the Hamiltonian $H$ always exist (see \cite[Chapter 9]{Bettiol-Vinter-book2024}).
	\item[(d)]  {The term `$\sum_{j\in I(S,\bar x(S))} \beta_j\,\nabla h_j(S,\bar x(S))$' of the transversality condition (iv) of Theorem \ref{pm} can be omitted when the endpoint constraint \(E\) is a closed subset of ${\rm Gr}\, C \times\R\times\R^n $. This is due to the fact that the initial constraint `$x(S)\in C(S)$' can be removed in the proof analysis. In this case we can deduce the following transversality condition:}
	\begin{itemize}
	\item[(iv)*:] there exist $\nu_0 \in N_{C(\bar S)}(\bar x (\bar S))$, $\nu_1 \in N_{C(\bar T)}(\bar x (\bar T))$,
    \[
    \begin{aligned}
    \xi_0 &\in 
    \operatorname{sub\text{-}ess}_{t\to\bar S} H(t,\bar x(\bar S),p(\bar S)) 
    \ - \ p(\bar S)\cdot \nu_0, \\
    \xi_1 &\in 
    \operatorname{super\text{-}ess}_{t\to\bar T} H(t,\bar x(\bar T),p(\bar T)) 
     \ - \ p(\bar T)\cdot \nu_1,
    \end{aligned}
    \]
    such that
    \[
    (-\xi_0,p(\bar S),\xi_1,-p(\bar T))
    \in
    \lambda\,\partial g(\bar S,\bar x(\bar S),\bar T,\bar x(\bar T))
    + N_E(\bar S,\bar x(\bar S),\bar T,\bar x(\bar T)).
		\]
	\end{itemize}
	\item[(e)]  {Consider the particular case when we have a fixed initial time $S_0$, we take $E=\{S_0\} \times E_1$ where $E_1 \subset \R^n \times \R \times \R^n$ is a closed set and $g=g_1(x_0,T,x_1)$. 
	We notice that the transversality condition (iv) of Theorem \ref{pm} can be simplified to 
	\begin{itemize}
	\item[(iv)**:] there exist $\nu_1 \in N_{C(\bar T)}(\bar x (\bar T))$ and
    \[
    \xi_1 \in 
    \operatorname{super\text{-}ess}_{t\to\bar T} H(t,\bar x(\bar T),p(\bar T)) \ -  \ p(\bar T)\cdot \nu_1,
    \]
    such that
    \[
    (p(S_0),\xi_1,-p(\bar T))
    \in
    \lambda\,\partial g_1(\bar x(S_0),\bar T,\bar x(\bar T))
    + N_{E_1}(\bar x( S_0),\bar T,\bar x(\bar T)).
		\]
	\end{itemize}
	Observe that the extra term `$\sum_{j\in I(S_0,\bar x(S_0))} \beta_j\,\nabla_x h_j(S_0,\bar x(S_0))$' can be now incorporated in the value of the costate arc of bounded variation  $p$.}
\end{itemize}
\end{rem}

\ \\

\begin{prop}[Normality - Measurable Time Dependence]\label{libre}
 {Let $([\bar S,\bar T],\bar x,\bar u)$ be a local minimizer for $(P)$ such that $\bar T - \bar S>0$.
Assume that {\rm (H1)}--{\rm (H5)} are satisfied and that $E=E_0\times\R^n$ (which implies that the state right endpoint is free), where $E_0$ is a closed subset of ${\rm Gr}\, C \times\R$.
Then, all the assertions of Theorem \ref{pm} are valid with 
\(
\lambda = 1.
\)}
\end{prop}

\ \\

\noindent
We now present a result establishing the validity of the necessary conditions obtained earlier when condition (H4)(iii) is replaced by a (not necessarily strict) positive linear independence constraint qualification that is in force together with a local prox-regularity property defined on a tube around a reference arc.

\begin{thm}\textbf{(Free End-Time Maximum Principle - Measurable Time Dependence - Local Hypotheses)}\label{pmo}
Let $([\bar S,\bar T], \bar x, \bar u)$ be a local minimizer for $(P)$ such that $\bar T - \bar S > 0$. 
The assertions of Theorem \ref{pm} and Proposition \ref{libre} remain valid when hypothesis {\rm (H4)} is replaced by 
\begin{itemize}
    \item [$(\widetilde{H4})$] conditions (i) and (ii) of (H4) are satisfied; moreover
    \begin{itemize}
      \item[(iii):] the set $C(t)$ satisfies the following positive linear independence constraint qualification: 
		\[
  \text{if} \;\; \sum_{j\in I(t,x)}\alpha_j\nabla_x h_j(t,x)=0 \; \text{with } \; \alpha_j \ge 0, \quad \text{then }  \alpha_j=0 \; \text{for all } j \in I(t,x), \, 
  \]
  for all $x \in \partial C (t)$, $t \in [\bar S- \bar \delta, \bar T +\bar \delta]$;	
  \item[(iv):] there exist constants $r>0$, $\delta_0\in (0,\bar \delta/2]$ and $\delta_1\in (0, \delta_0/4]$, and a Lipschitz continuous arc $\tilde y \in W^{1,\infty}([\bar S- \bar \delta, \bar T +\bar \delta],\R^n)$ such that $\| \bar x_e - \tilde y\|_{L^{\infty}} < \delta_1$ and $C(t)\cap \left(\tilde y (t) + \delta_0 \B\right)$ is uniformly $r$-prox-regular for all $t \in [\bar S- \bar \delta, \bar T +\bar \delta]$.
\end{itemize}
\end{itemize}
\end{thm}

 \ \\

\begin{rem}
The local uniform prox-regularity condition $(\widetilde{H4})$ (iv) is always satisfied when the set $C(t)$ is uniformly prox-regular (as in \cite{ChamZeid}) and conditions $(\widetilde{H4})$ (i)-(iii) hold. Indeed, in these circumstances, the minimizer is Lipschitz continuous and, by using the property established in \cite[Remark 3.5]{ChamZeid}, one can take $\tilde{y} = \bar{x}$ and thereby obtain $(\widetilde{H4})$ (iv). We observe that the boundedness of the controlled dynamics function $f$ (see hypothesis (H3)(ii)) alone does not guarantee the Lipschitz regularity of the reference minimizer arc $\bar x$ of problem (P) (as it would obviously happen for a `classical' optimal control problem); this is due to the presence of the normal cone to the moving set $C(\cdot)$ in the dynamic constraint. Lipschitz regularity of $\bar x$ can be proved by imposing additional assumptions on the moving set $C(\cdot)$, such as (global) uniform prox-regularity (cf. the regularity results and related discussions in \cite{N_Th,ChamZeid}). Finally, notice also that condition (H4) implies the validity of the local constraint qualification $(\widetilde{H4})$ when some additional hypotheses are imposed, such as global hypomonotonicity conditions on the gradients of the functions $h_j$'s, which would guarantee uniform prox-regularity properties of $C(\cdot)$ (the latter can be derived using results such as \cite[Theorem 3.1]{ANTh}).
\end{rem}

\ \\

\noindent 
{\bf Example 1} 
\vspace{0.05 in}

\noindent
This example is a variation (adapted to the sweeping process framework considered in our paper) of a simple optimal control problem in resource economics proposed in \cite{Bettiol-Vinter-SICON2025}, where the purpose was to highlight the distinction between the necessary conditions introduced in \cite{Bettiol-Vinter-SICON2025} and the previously established ones (as discussed in detail in \cite{Bettiol-Vinter-SICON2025}). Here, the objective is to illustrate the applicability of the new necessary conditions in the sweeping process setting provided by Theorem \ref{pm} and Proposition \ref{libre} for the measurable time dependent case, which differ from results that can be obtained in the presence of an autonomous dynamics function $f$, in particular with regard to the transversality condition at time $\bar T$ (cf. earlier results established for free end-time sweeping process problems \cite{CMNN2024}, where only the autonomous case was considered, and Corollary \ref{cor} below).

\[ 
(FT1) \hspace{0.4 in} \left\{
\begin{array} {l}
\mbox{Minimize } - \int_0^T \Big( x(t)-at -b(t)u(t) \Big) dt - 3 x(T) \\
\mbox{over $T\in [0,3]$, meas. $u:[0,T]\rightarrow \R$ and $x  \in W^{1,1}(0,T)$ s.t.} \\
\dot{x}(t)\in u(t) -N_C(x(t)), \mbox{ a.e. } t \in [0,T], \\
u(t) \in [0, 1], \mbox{ a.e. } t \in [0,T],
\\
x(0)=0\,.
\end{array}
\right.
\]
Here, $C:=\{x\le 1\}$, $a\in(\frac{5}{4},\frac{11}{8})$ and
$$
b(t) :=
\left\{
\begin{array}{ll}
5 &\mbox{for } -\infty < t \leq 1
\\
\frac14 &\mbox{for } 1 < t < 2\\
 4 &\mbox{for } 2 \leq t < \infty\,.
\end{array}
\right.
$$
A standard state augmentation technique allows us to substitute the integral term in the cost of (FT1) with an auxiliary state variable evaluated at the end time $T$. This leads to a free end-time problem with discontinuous (measurable) time-dependent data of the form (P). Observe that all the hypotheses for the validity of the maximum principle of Proposition \ref{libre} are satisfied; therefore, we can apply the necessary conditions of optimality in the normal form (i.e., with a unit cost multiplier). 
In this case, the Hamiltonian is
$$
{\cal H} (t,x,p,u) = (p- b(t))u  -at +x 
$$
and the maximized Hamiltonian is $H(t,x,p) = (p- b(t))^+   -at + x $.
Let $(\bar T, \bar x, \bar u, p)$ be an extremal, together with the associated Radon measure $\mu$ and the bounded measurable function $\eta$ on $[0, \bar T]$. We have 
\begin{equation}
\label{ex0}
\dot{\bar x}(t) = \bar u(t) - \eta(t),
 \qquad\text{and } \quad 
 \eta(t) > 0 \Rightarrow \bar x(t)=1, \; \mbox{a.e. } t\in [0, \bar T].
\end{equation}
The costate arc $p$ is a solution to the differential equation $dp(t) = d\mu(t) -1$ on $[0, \bar T]$ such that $p(\bar T)=3$. 
 The Weierstrass condition tells us that 
\begin{equation}
\label{ex1}
\bar u(t)\in
\left\{
\begin{array}{ll}
\{1 \}&\mbox{if }  p(t)-b(t)>0
\\
\left[0,1\right] &\mbox{if }  p(t)-b(t)=0
\\
\{0 \} &\mbox{if }  p(t)-b(t)<0.
\end{array}
\right.
\end{equation}
We now turn to examine the free-time transversality condition at time $\bar T$, that yields
$$
 N_{[0,3]}(\bar T) \cap \Big(\underset{t \rightarrow \bar T}{\mbox{super-ess}}\, H(t,\bar x(\bar T),p(\bar T))  -  p(\bar T) \times  \nu \Big)  \neq \emptyset , \mbox{ for some } \nu \in  N_{C}(\bar x (\bar T)).
$$
 We have
 $$
 \underset{t \rightarrow \bar T}{\mbox{super-ess}}\, H(t,\bar x(\bar T),p(\bar T)) =
 \left\{
 \begin{array}{ll}
 \{\bar x(\bar T)-a\bar T\} & \mbox{if } \bar T \in (0,1)
 \\
 \{\bar x(\bar T)-a\bar T\}\cup\{\bar x(\bar T) + 2.75 - a\bar T\} & \mbox{if } \bar T = 1
 \\
  \{\bar x(\bar T) + 2.75 - a\bar T \}  & \mbox{if } \bar T \in (1,2)
   \\
\left[\bar x(\bar T) - a\bar T ,\bar x(\bar T) + 2.75 - a\bar T \right]
 & \mbox{if } \bar T = 2
 \\
  \{ \bar x(\bar T) - a\bar T\} & \mbox{if } \bar T \in (2,3] 
 \end{array}
\right.
$$ 
We deduce that necessarily $\bar T \in [1,2]$. 
The complementary slackness condition `$\eta(t) p(t)=0$ a.e. on the set $\{ t: \bar x(t) =1 \}$', combined with the Weierstrass condition and \eqref{ex0} implies that it is not possible to have `$\bar u(t)>0$' for $t\in {\cal E}$, where ${\cal E}\subset \{ s: \bar x(s) =1 \}$ is a set of positive measure. Moreover, reasoning further using information provided by the Weierstrass condition and the costate differential equation, we obtain that $\bar x(\bar T)=1$ and $\bar x(t)<1$ when $t< \bar T$. Observe also that all the extremals are admissible processes also for a modified version of problem (FT1) in which $C=\R$. Reasoning as above (now taking $C=\R$) reveals that the unique (global) minimizer is $(\bar T=2, \bar x, \bar u)$, where
\begin{equation}
\label{ex_2}
(\bar u(t), \bar x (t))=
\left\{
\begin{array}{ll}
(0,0) &\mbox{for }  t \in [0,1) 
\\
(1, t -1) &\mbox{for }  t \in [1, \bar T=2]\ .
\end{array}
\right.
\end{equation}
Observe that the free-time transversality condition at time $\bar T=2$ provided for the autonomous case in \cite[Theorem 5.2]{CMNN2024} (cf. Corollary \ref{cor} below) would become, in the present case,
$$
\frac{1}{\bar T}\int_0^{\bar T}[(p(t)-b(t))\bar u (t) + \bar x(t) -at] dt\in N_{[0,3]}(\bar T=2)= \{0\}.
$$
However, it is straightforward to see that this condition is never satisfied in this example, regardless of the choice of the multiplier $p$ (provided it is compatible with the minimizer $(\bar T=2, \bar x, \bar u)$ described above). This highlights the fact that, also for sweeping process problems, the necessary conditions obtained in the autonomous case must be revised, and it is crucial to consider different (more refined) necessary conditions involving the notions of  sub- or super-essential values of the maximized Hamiltonian.

\begin{rem}
Notice that, in the example proposed in \cite{Bettiol-Vinter-SICON2025} the term `$at$' in the integral is replaced by `$2$'. A similar analysis can be employed to derive analogous properties for the unique (now local) minimizer $(\bar T=2, \bar x, \bar u)$, which coincides with the minimizer described in \eqref{ex_2}. Observe in particular that the free-time transversality condition at time $\bar T=2$ provided for the autonomous case in \cite[Theorem 5.2]{CMNN2024} is never satisfied in this example as well.
\end{rem}

\noindent {\bf Example 2}
\vspace{0.05 in}

\noindent
Consider the following optimal control problem

\[ (FT2) \hspace{0.4 in} \left\{
\begin{array} {l}
\mbox{Minimize } T - 2x(T) \\
\mbox{over intervals } [0,T], \mbox{ arcs }
x  \in W^{1,1}([0,T];\R) \\
\mbox{\hspace{0.5 in}and measurable functions }
u :[0,T]\rightarrow \R \mbox{ satisfying } \\
\dot{x}(t) \in u(t) -N_{(-\infty,1]}(x(t)) \mbox{ a.e. } t \in [0,T], \\
u(t) \in [0,1] \mbox{ a.e. } t \in [0,T],
\\
S=0,\,x(S)=0.
\end{array}
\right.
\]
It is immediate to see that the optimal solution is given by 
\[
\bar T=1,\qquad \bar x(t)=t,\qquad \bar u(t)=1.
\]
Assume that the terms $p(\bar S)\cdot\nu_0$ and $p(\bar T)\cdot\nu_1$ in  condition (iv) of Theorem \ref{pm} are replaced by the essential value sets mentioned in Remark \ref{nc}, and let us derive a contradiction.  
Then, applying Proposition \ref{libre}, there exist $\xi_1\in \underset{t\to1}{\operatorname{super\text{-}ess}} \left[\sup_{u\in[0,1]} p(1)\,u\right]$ and 
$p\in NBV([0,1];\R)$ such that
\[
\xi_1 = 1 =\frac12\, p(1).
\]
On the other hand, the inclusion $\xi_1\in \underset{t\to1}{\operatorname{super\text{-}ess}}\left[\sup_{u\in[0,1]} p(1)\,u\right]$ yields 
\[
\xi_1
   = \sup_{u\in[0,1]} p(1)\,u
   = p(1),
\]
which provides the desired contradiction.

\subsection{The Lipschitz Time-dependent Case}

We now proceed to provide necessary conditions for the Lipschitz time-dependent case. 
For a local minimizer $([\bar S, \bar T],\bar x, \bar u)$ of $(P)$, with $\bar S < \bar T$, we shall assume that hypotheses (H1),(H2)$'$,(H3)$'$, (H4)$'$ and {\rm (H5)} are satisfied, for some $\bar \delta>0$. Here, the stronger conditions (H2)$'$, (H3)$'$ and (H4)$'$ replace (H2), (H3) and (H4), previously considered in the measurable time-dependent case:
\begin{itemize}
\item[(H2)$'$]  {$U$ is a measurable subset of $\R^m$ (independent of $t$);}

\item[(H3)$'$] hypothesis (i) of (H3) hold, and in addition
\begin{itemize}
  \item[$(ii)'$]  there exist $k_f \in L^1(\bar S, \bar T )$ and $c_f>0$ such that
  \[
  |f(t',x',u)-f(t'',x'',u)|\le k_f(t)(|x'-x''|+|t'-t''|),\qquad |f(t,x',u)|\le c_f,
  \]
  for all $(t'',x''),(t',x')\in (t,\bar x (t)) + \bar \delta \B$ and all $u\in U$, a.e. $t\in [\bar S, \bar T ] $; 
\end{itemize}
\item[(H4)$'$] hypotheses (i), (ii) and (iii) of (H4) hold, and in addition  $\nabla_{t}\nabla_{x} h_j(t,x)$ exists and is continuous with respect to $(t,x)$ on the `tube' $\{(t',x')\in (t,\bar x (t)) + \bar \delta \B \,:\, t\in [\bar S, \bar T ] \}$.
\end{itemize}

\ \\

\begin{thm}\label{pml}
\textbf{(Free End-Time Maximum Principle - Lipschitz Time Dependence)}
Let $([\bar S,\bar T],\bar x,\bar u)$ be a local minimizer for $(P)$ such that $\bar T - \bar S>0$.
Assume that (H1), (H2)$'$, (H3)$'$, (H4)$'$ and (H5) hold for some $\bar \delta>0$.
Then, there exist a couple of functions $(a,p) \in NBV([\bar S,\bar T];\R \times \R^n)$, a scalar $\lambda\ge 0$, finite signed Radon measures $\dd\mu_j$ on $[\bar S,\bar T]$, and bounded measurable functions $\eta_j:[\bar S,\bar T]\to [0,+\infty)$ for $j=1,\dots,J$, such that 
\begin{itemize}
    \item[(i):] $(a,p,\lambda)\neq 0$;
    
    \item[(ii):] $\dot{\bar x}(t)= f(t,\bar x(t),\bar u(t))- \sum_{j=1}^J \eta_j(t)\,\nabla_x h_j(t,\bar x(t)), \qquad \eta_j(t)\neq0 \Rightarrow h_j(t,\bar x(t))=0, \text{ a.e. } t\in[\bar S,\bar T] $;
    
    \item[(iii):] 
		$\displaystyle(\dd a,-\dd p) \in \operatorname{co}\,\partial_{t,x}\mathcal H(t,\bar x(t),p(t),\bar u(t))\,\dd t$
		 
	\hspace{0.2 in}	$\displaystyle-\sum_{j=1}^{J}\Big(\nabla_t h_j(t,\bar x(t)), \nabla_x h_j(t,\bar x(t))\Big)\,\dd\mu_j(t) - p(t)\sum_{j=1}^{J}\eta_j(t)\, \Big( \nabla_{t}\nabla_{x} h_j(t,\bar x(t)) ,D_x^2 h_j(t,\bar x(t))\Big)\,\dd t$;
		
    \item[(iv):]     
    $\displaystyle
    (-a(\bar S),p(\bar S),a(\bar T),-p(\bar T))
    \in
    \lambda\,\partial g(\bar S,\bar x(\bar S),\bar T,\bar x(\bar T))
    + N_E(\bar S,\bar x(\bar S),\bar T,\bar x(\bar T)) $

		\hspace{3.2 in} $
		\displaystyle + \sum_{j\in I(\bar S,\bar x(\bar S))} \beta_j\,\nabla h_j(\bar S,\bar x(\bar S))\times \{0\} \times \{0\} $
		
		for some $\beta_j\ge 0$,  $j\in I(\bar S,\bar x(\bar S))$;
    
		\item[(v):]
$\displaystyle {\cal H} (t,\bar{x} (t), p(t), \bar{u} (t))
\;=\; \max_{u \in U} {\cal H} (t, \bar{x} (t), p(t), u) \quad
\mbox{a.e.  }t \in [\bar{S},\bar{T}] $;
\item[(vi):]
${\cal H}(t, \bar{x} (t), p(t), \bar{u} (t)) \; =\; a(t) \quad
\mbox{a.e. } t \in [\bar{S},\bar{T}]$;
    
    \item[(vii):] $\eta_j(t)\,p(t)\cdot\nabla_x h_j(t,\bar x(t))=0$,
    \quad for each $j\in I(t,\bar x(t))$, a.e.\ $t\in[\bar S,\bar T]$;
    
    \item[(viii):] $\operatorname{supp}(\dd\mu_j)\cup \operatorname{supp}(\eta_j) \subset \{\,t\in[\bar S,\bar T]: h_j(t,\bar x(t))=0\,\}$, \quad for each $j=1,\dots,J$;
    
    \item[(ix):] $p(t)\cdot\nabla_x h_j(t,\bar x(t))\,\dd\mu_j(t)$
    \text{is a nonnegative measure}, \quad for each $j=1,\dots,J$.
\end{itemize}
\end{thm}

Similarly as the measurable time dependent case we have also for Lipschitz time dependent case conditions under which the maximum principle applies in the normal form.

\begin{prop}[Normality - Lipschitz Time Dependence]\label{libre1}
 {Let $([\bar S,\bar T],\bar x,\bar u)$ be a local minimizer for $(P)$ such that $\bar T - \bar S>0$.
Assume that (H1), (H2)$'$, (H3)$'$, (H4)$'$ and (H5) hold for some $\bar \delta>0$, and that $E=E_0\times\R^n$ (which implies that the state right endpoint is free), where $E_0$ is a closed subset of ${\rm Gr}\, C \times\R$).
Then, all the assertions of Theorem \ref{pml} are valid with 
\(
\lambda = 1.
\)}
\end{prop}

It is immediate to see that when the dynamics function $f$ is independent of the time variable $t$, since the Hamiltonian ${\cal H}$ becomes independent of $t$ as well, the costate inclusion (iii) of Theorem \ref{pml} becomes simpler and only the convex hull of the partial limiting subdifferential of the Hamiltonian $\operatorname{co}\,\partial_{x}\mathcal H$ plays a role. If in addition also $C$ is independent of $t$, then from the costate inclusion (iii), the Weierstrass condition (v), condition (vi) and the complementary slackness condition (vii) of Theorem \ref{pml} we deduce that
\begin{equation}\label{avH}
\int_{\bar S}^{\bar T} p(t) \cdot \dot{\bar x}(t) \ dt 
= \int_{\bar S}^{\bar T} {\cal H} (\bar{x} (t), p(t), \bar{u} (t)) \ dt  = \bar a \, (\bar T-\bar S),
\end{equation}
for some constant $\bar a\in \R$. 
This is summarized in the following corollary.

\begin{cor}\label{cor}
\textbf{(Free End-Time Maximum Principle - Case when $f=f(x,u)$)}
Let $([\bar S,\bar T],\bar x,\bar u)$ be a local minimizer for $(P)$ such that $\bar T - \bar S>0$.
Assume that (H1), (H2)$'$, (H3)$'$, (H4)$'$ and (H5) hold for some $\bar \delta>0$. Suppose in addition that $f$ does not depend on $t$: $f=f(x,u)$. 
Then, there exist a couple of functions $(a,p) \in NBV([\bar S,\bar T];\R \times \R^n)$, a scalar $\lambda\ge 0$, finite signed Radon measures $\dd\mu_j$ on $[\bar S,\bar T]$, and bounded measurable functions $\eta_j:[\bar S,\bar T]\to [0,+\infty)$ for $j=1,\dots,J$, such that conditions (i)-(ix) of Theorem \ref{pml} are satisfied and (iii) becomes:
\item[(iii)$'$:] 
		$\displaystyle -\dd p \in \operatorname{co}\,\partial_{x}\mathcal H(\bar x(t),p(t),\bar u(t))\,\dd t-\sum_{j=1}^{J}\nabla_x h_j(t,\bar x(t))\,\dd\mu_j(t) - p(t)\sum_{j=1}^{J}\eta_j(t)\, D^2_x h_j(t,\bar x(t))\,\dd t$, 
    
		$\displaystyle\dd a = -\sum_{j=1}^{J}\nabla_t h_j(t,\bar x(t))\,\dd\mu_j(t) - p(t)\sum_{j=1}^{J}\eta_j(t)\,\nabla_{t}\nabla_{x} h_j(t,\bar x(t))\,\dd t$.

\noindent	
If, in addition, also $C$ is independent of $t$, then we obtain also condition \eqref{avH} for some constant $\bar a\in \R$ such that $\bar a \; =\; {\cal H}(\bar{x} (t), p(t), \bar{u} (t)) $
a.e. $t \in [\bar{S},\bar{T}]$.
\end{cor}

\ \\

\begin{thm}\textbf{(Free End-Time Maximum Principle - Lipschitz Time Dependence - Local Hypotheses)}\label{pmo1}
Let $([\bar S,\bar T], \bar x, \bar u)$ be a local minimizer for $(P)$ such that $\bar T - \bar S > 0$. 
The assertions of Theorem \ref{pml}, Proposition \ref{libre1} and Corollary \ref{cor} remain valid when hypothesis {\rm (H4)} is replaced by $(\widetilde{H4})$ of Theorem \ref{pmo}.
\end{thm}

\section{Preliminary Results}

 {
Suppose that $([\Bar{S}, \Bar{T}],\Bar{x},\Bar{u})$ is a local minimizer for $(P)$. Suppose that (H1)-(H5) are satisfied for some $\bar \delta >0$ and, reducing the size of $\bar \delta >0$ if necessary, that  
$$
g(S,x(S),T,x(T)) \geq g(\bar{S},\bar{x}(\bar{S}),\bar{T},\bar{x}(\bar{T}))
$$
for every admissible process $([S,T],x,u)$ for ($P$) such that
\begin{equation*}
|S - \bar S  |    +  |T - \bar T | + ||x_e - \bar x_e ||_{L^\infty([S\wedge \bar S,T \vee \bar T],\R^n)} \leq \bar \delta.
\end{equation*}
}

\noindent
We introduce the scalar function $h: \R \times \R^n \to \R$ by
\[
h(t,x) := \frac{1}{3}\sum_{j=1}^J h_j^+(t,x)^3, 
\]
where the family of functions $(h_j)_{1\leq j \leq J}$ is given by {\rm (H4)} (i).

\ \\

\noindent
{\bf A reduction technique.} 
Observe that assumptions (H3) can be replaced by the stronger hypothesis, written (H3)*, in which the conditions are satisfied globally w.r.t. the $x$ variable (i.e. not only in a ``tube'' around the reference arc $\bar x (.) $). This can be done employing a (well-known) ``truncation'' technique (cf. \cite{Bettiol-Vinter-book2024}). 

\begin{lem}\label{ingt}
Suppose that {\rm (H4)} is satisfied. Then, there exist some constants \(\gamma,\eta>0\) such that 
\[
\bigl|\nabla_x h(t,x)\bigr|\ge \gamma\sum_{j=1}^J h_j^+(t,x)^2, 
\]
for all $(t,x) \in \{ (s,y) \; : \; s\in [\bar{S}-\bar \delta,\bar{T}+\bar \delta], \;   h (s,y) \le \eta \} $. 
\end{lem}
\begin{proof} Take $\eta\in (0,\sqrt[3]{\eta_0}/3)$. 
We argue by contradiction: suppose there exist sequences $\gamma_i \downarrow 0$,  $(t_i)_{i \in \N} \subset [\bar{S}-\bar\delta,\bar{T}+\bar\delta]$ and $x_i $ such that $h(t_i,x_i) \le \eta$ and 
\[
\bigl|\nabla_x   h(t_i,x_i)\bigr|< \gamma_i\sum_{j=1}^{J} h_j^+(t_i,x_i)^2.
\]
Then, we would have
\[
\left| \sum_{j=1}^{J} \frac{h_j^+(t_i,x_i)^2}{\sum_{j=1}^{J} h_j^+(t_i,x_i)^2} \nabla_x h_j(t_i,x_i) \right| \to 0 \text{ as } i \to +\infty,
\]
which contradicts condition {\rm (H4)} (iii).
\end{proof}

\noindent
Notice that it is not restrictive to assume that $\nabla_x h_j(t,x) = 0$ for all $(t,x)$ such that $h_j(t,x) \leq -\eta$, for all $j = 1, \dots, J$. Indeed, otherwise, using a standard regularization argument (one can apply, for instance, the regularization technique in \cite{ChamZeid}), we can replace $h_j$ with a function $\hat{h}_j$ that coincides with $h_j$ when $h_j(t,x) \geq -\frac{\eta}{2}$, satisfies $\nabla_x \hat{h}_j(t,x) = 0$ when $h_j(t,x) \leq -\eta$, and is such that (H5) remains in force.
Finally, by adjusting the size of $\rho > 0$ if necessary, we may suppose that the following strengthened condition holds:
\begin{itemize} 
\item[(H5)*:] For all $R>0$ there exist \(\rho\in(0,1)\) 
such that for each \(k\in I_\eta^+(t,x)\), we have
\[
\sum_{\substack{j\in I_\eta(t,x)\\ j\neq k}}
\bigl|\nabla_x h_k(t,x)\cdot\nabla_x h_j(t,x)\bigr|
\le \rho\,\bigl|\nabla_x h_k(t,x)\bigr|^2,
\]
for all $x \in R\B$ and all $t \in [\bar S- \bar \delta, \bar T +\bar \delta]$.
\end{itemize}
Here, $I_\eta(t,x)$ denotes the set of active indices `near the boundary'
defined for all $(t,x) \in \R\times \R^n$ by
\[
I_\eta(t,x):=\left\{\,j\in\{1,\dots,J\}:\, |h_j(t,x)| \le \eta\,\right\}.
\]
Recall that $I_\eta^+(t,x)$ denotes the set of active indices `near the boundary from the outside': 
\[
I_\eta^+(t,x):=\left\{\,j\in\{1,\dots,J\}:\, 0\le h_j(t,x) \le \eta\,\right\}.
\]

\begin{lem}\label{eqn}
There exist constants \(k_1, k_2 > 0\) such that, for all 
\((v_1,\dots,v_J) \in [0,+\infty)^{J}\),
\[
k_1 \left(\sum_{j=1}^{J} v_j^3 \right)^{2/3} 
\leq
\sum_{j=1}^{J} v_j^2
\leq
k_2 \left(\sum_{j=1}^{J} v_j^3 \right)^{2/3}.
\]
\end{lem}

\begin{proof}
The result follows immediately from the equivalence of norms in finite-dimensional spaces.
\end{proof}

\begin{lem}\label{kap}  
Assume that {\rm (H4)} holds. Take $R>0$. Let $\varepsilon >0$, then there exists $\kappa_{\varepsilon}>0$ such that, for all $(t,x) \in [\bar{S}-\delta,\bar{T}+\delta] \times R\B$,
\[
h(t,x) \leq \kappa_{\varepsilon} \implies d_{C(t)}(x) \leq \varepsilon.
\]
\end{lem}

\begin{proof}
From {\rm (H4) (i),(iii)} we deduce that the function $d_{C(\cdot)}(\cdot)[\bar{S}-\delta,\bar{T}+\delta] \times R\B \to \R$ is continuous.
The result follows directly from the compactness of $[\bar{S}-\bar \delta,\bar{T}+\bar \delta] \times R\B$.
\end{proof}

\begin{lem}\label{phm}
Suppose that {\rm (H4)} is satisfied. Take $R>0$. Then, there exists a constant $\Lambda>0$ such that
\[
d_{C_j(t)}(x) \le \Lambda\, h_j^+(t,x),
\]
for all $(t,x)\in[\bar{S}-\bar \delta,\bar{T}+\bar \delta]\times R\B$ and for all $j\in\{1,\dots,J\}$, .
\end{lem}

\begin{proof}
We prove the claim of the lemma by contradiction. Suppose that there exist $j\in\{1,\dots,J\}$, and, for all $i \in \N$, $(t_i,x_i) \in [\bar{S}-\bar\delta,\bar{T}+\bar\delta] \times  R\B$ such that 
\begin{equation}\label{est0}
    d_{C_j(t_i)}(x_i) > i\, h_j^+(t_i,x_i).
\end{equation}
By a standard compactness argument, we may extract a subsequence (we do not relabel) such that $(t_i,x_i)\to (t,x)\in [\bar{S}-\bar\delta,\bar{T}+\bar\delta] \times  R\B$ as $i \to +\infty$.
From \eqref{est0} we deduce
\[
h_j^+(t_i,x_i)=h_j(t_i,x_i) 
\longrightarrow 0,\quad \text{ as } i \to +\infty,
\]
and so $x\in \partial C_j(t)$. It follows, from continuity of the function $d_{C(\cdot)}(\cdot)$, that
$
d_{C_j(t_i)}(x_i)\to 0
$  as $i \to +\infty$.
{From now on, we assume $i$ to be large enough so that $d_{C_j(t_i)}(x_i)\le r$, such that the set $C_j(t) \cap R\B$ is uniformly $r$-prox-regular (this is an immediate consequence of (H4) and standard arguments, cf. \cite[Proposition 2.2]{BCG2025}), for all $t \in [\bar{S}-\bar\delta,\bar{T}+\bar\delta]$.}
By the mean value theorem, there exists a point
$
y_i \in [x_i,\ \pi_{C_j(t_i)}(x_i)]
$
such that
\[
h_j(t_i,x_i)-h_j(t_i,\pi_{C_j(t_i)}(x_i))
= \nabla_x h_j(t_i,y_i)\cdot (x_i-\pi_{C_j(t_i)}(x_i)).
\]
Therefore, dividing across by $d_{C_j(t_i)}(x_i)>0$, we obtain 
\begin{align*}
\frac{h_j(t_i,x_i)}{d_{C_j(t_i)}(x_i)}
&= \frac{h_j(t_i,x_i)-h_j(t_i,\pi_{C_j(t_i)}(x_i))}{d_{C_j(t_i)}(x_i)} \\
&= \nabla_x h_j(t_i,y_i)\cdot 
\frac{x_i-\pi_{C_j(t_i)}(x_i)}{|x_i-\pi_{C_j(t_i)}(x_i)|} \\
&= \nabla_x h_j(t_i,y_i)\cdot
\frac{\nabla_x h_j(t_i,\pi_{C_j(t_i)}(x_i))}{|\nabla_x h_j(t_i,\pi_{C_j(t_i)}(x_i))|}
\end{align*}
Passing to the limit as $i \to +\infty$, we find 
\[
\frac{h_j(t_i,x_i)}{d_{C_j(t_i)}(x_i)} \to |\nabla_x h_j(t,x)|.
\]
But from \eqref{est0} we deduce that
\[
\underset{i \to +\infty}{\lim} \frac{h_j(t_i,x_i)}{d_{C_j(t_i)}(x_i)} = 0,
\]
that yields $|\nabla_x h_j(t,x)|=0$, in contradiction of {\rm (H4) (ii)}.
\end{proof}

\subsection{A Sweeping Process Approximation}
For a given real number $\alpha >0$, we shall consider the following control system
\begin{equation*}
    (CS_\alpha)\; \left\{
    \begin{aligned}
        &\dot{x}(t) = f(t,x(t),u(t)) - \dfrac{1}{\alpha}\,\nabla_x  h(t,x(t)), \text{ a.e. } t \in [S,T], \\
        & u(t) \in U(t), \text{ a.e. } t \in [S,T], \\
        & |S-\bar S|+|T - \bar T|+ \|x_e -\bar x_e\|_{L^\infty([\bar{S}-\bar\delta,\bar{T}+\bar\delta],\R^n)}\leq \bar \delta.
    \end{aligned}
    \right.
\end{equation*}
Similarly as before, an admissible process for $(CS_\alpha)$ is a triple $([S,T],x,u)$, where $[S,T]$ is an interval, $x \in W^{1,1}([S,T],\R^n)$ and $u:[S,T] \to \R^m$ is a measurable function such that all the constraints of $(CS_\alpha)$ are satisfied.

\begin{lem}\label{cvg} 
Suppose that (H2), (H3)* and {\rm (H4)} are satisfied. 
There exists a constant $A>0$ such that the following holds. 
Take some sequences of real numbers $(\alpha_i)_{i \in \N}, (S_i)_{i \in \N}, (T_i)_{i \in \N}$, a sequence of points $(x_0^i)_{i \in \N}$ and a sequence of admissible controls $(u_i)_{i \in \N}$, such that
\[
\alpha_i\downarrow 0, 
\qquad |S_i-S|+ |T_i-T|+ |x_0^i - x_0| \to 0,
\qquad d_{\mathcal{E}}(u_i,u) \to 0,
\qquad \text{ as } i \to +\infty,
\]
\[
\qquad |x_0^i - \bar x_e (S_i)|+ |S_i- \bar S|+ |T_i- \bar T| \leq \bar \delta, 
\qquad x_0^i\in C(S_i),
\]
for some $S,T \in \R$, $S<T$, $x_0 \in \R^n$ and some control $u$. Then, 
\begin{itemize}
    \item[(i):] there exists $i_1\in\N$ such that, for all $i\ge i_1$, the control system $(CS_{\alpha_i})$ admits a unique solution $x_{\alpha_i}\in W^{1,1}([S_i,T_i];\R^n)$, with $x_{\alpha_i}(S_i)=x_0^i$, such that $([S_i,T_i],x_{\alpha_i},u_i)$ is an admissible process for $(CS_{\alpha_i})$, and $x_{\alpha_i}$ satisfies
    \begin{equation}\label{est1}
    h(t,x_{\alpha_i}(t))\le A\,\alpha_i^{3/2},\quad \text{for all } t\in[S_i,T_i], \; \text{ and }
    \end{equation}
    \item[(ii):] the sequence $((x_{\alpha_i})_e)_{i \in \N}$ converges in $L^\infty$ and weakly in $W^{1,2}$
    as $i\to+\infty$ to a function $\tilde x$, which is the extension of an arc 
    $x\in W^{1,\infty}([S,T];\R^n)$ that is the unique solution in
    $W^{1,1}([S,T];\R^n)$ to the sweeping process
    \[
    (SP1) \quad
    \begin{cases}
    \dot{x}(t)=f(t,x(t),u(t)) - N_{C(t)}(x(t))
    & \text{a.e. } t\in[S,T],\\
    x(S)=x_0 \in C(S).
    \end{cases}
    \]
\end{itemize}
\end{lem}

\begin{proof}
{\bf (i).} The conditions on the sequences $(x_0^i)_{i \in \N}$ and $(u_i)_{i \in \N}$ ensure that $x_0 \in C(S)$ and that $u$ is an admissible control. 
Consider the constants $K_h$, $c_f$, $\eta$, $\gamma$, $k_1$ and $k_2$ provided by assumptions (H3)* and (H4) and Lemmas \ref{ingt} and \ref{eqn}. Set
\[
M_0:=3^{2/3} k_2 K_{h}(1+c_f) \quad \text{ and } \quad A_0:= \frac{4M_0}{(3^{2/3}{\gamma} k_1)^2}.
\]
For all $i \in \N$, there exists a unique maximal solution $x_{\alpha_i}$ to
\[
\dot{x}(t) = f(t,x(t),u_i(t)) - \dfrac{1}{\alpha_i}\,\nabla_x {h}(t,x(t)), \text{ a.e. } t \in [S_i,T_i],
\]
such that $x(S_i) = x_0^i$ (this solution is not necessarily defined on the whole interval).
Since ${h}(S_i,x_{\alpha_i}(S_i)) \leq 0$, there exists $i_1 \in \N$, such that 
$\alpha_{i_1} \leq \frac{{\eta}^{2/3}  }{A_0}$
and, for all $i \geq i_1$, we may define
\[
\tau_i := \sup \left\{ t \in [S_i, T_i] :\; \forall s \in [S_i,t], \, h(s, x_{\alpha_i}(s))^{2/3} \leq A_0 {\alpha_i} \right\},
\]
note that $\tau_i > S_i$. Take $i \geq i_1$ and assume for the sake of contradiction that $\tau_i < T_i$.
Then, by continuity, ${h}(\tau_i,x_{\alpha_i}(\tau_i))^{2/3} = A_0{\alpha_i}$ and there exist $\nu >0$ such that, for all $t \in [\tau_i-\nu,\tau_i]$, ${h}(t,x_{\alpha_i}(t))^{2/3} \geq \frac{A_0{\alpha_i}}{2}$. But, for a.e. $t \in [\tau_i-\nu,\tau_i]$,
\begin{align*}
    \frac{d}{dt} ({h}(t,x_{\alpha_i}(t)))
    =& \nabla_t {h}(t,x_{\alpha_i}(t)) + \nabla_x {h}(t,x_{\alpha_i}(t)) \cdot \dot{x}_{\alpha_i}(t) \\
    =& \nabla_t {h}(t,x_{\alpha_i}(t)) - \frac{1}{{\alpha_i}} |\nabla_x {h}(t,x_{\alpha_i}(t))|^2 + \nabla_x {h}(t,x_{\alpha_i}(t)) \cdot f(t,x_{\alpha_i}(t),u_i(t)).\\
\end{align*}
From Lemmas \ref{ingt} and \ref{eqn} (and bearing in mind (H3)*) we obtain 
\begin{align*}
    \frac{d}{dt} ({h}(t,x_{\alpha_i}(t))) \leq & 3^{2/3} k_2 K_{h}(1+c_f) {h}(t, x_{\alpha_i}(t))^{2/3} - \frac{(3^{2/3}{\gamma} k_1)^2}{{\alpha_i}}{h}(t, x_{\alpha_i}(t))^{4/3} \\
    \leq & {h}(t, x_{\alpha_i}(t))^{2/3}\left(M_0 - \frac{(3^{2/3}{\gamma} k_1)^2}{{\alpha_i}}{h}(t, x_{\alpha_i}(t))^{2/3}\right) \\
    \leq & A_0 {\alpha_i} \left(M_0 - {(3^{2/3}{\gamma} k_1)^2}\frac{A_0}{2}\right) \\
    = & -A_0 M_0 {\alpha_i} \\
    < & 0, \qquad \text{ a.e. } t \in [\tau_i-\nu,\tau_i].
\end{align*}
Hence ${h}(\tau_i,x_{\alpha_i}(\tau_i)) < {h}(\tau_i-\nu,x_{\alpha_i}(\tau_i-\nu)) \leq (A_0 {\alpha_i})^{3/2}$. A contradiction. 
It follows that $\tau_i = T_i$, which proves property (i) of the lemma statement, with $A:=A_0^{3/2}$. 

\ \\

\noindent
{\bf (ii).} We now study the convergence of $((x_{\alpha_i})_e)_{i \in \N}$.
Write $A_1:=3^{2/3}k_2A_0$.  
We obtain
\[
    |\dot{x}_{\alpha_i}(t)| \leq c_f+ A_1K_{h}, \quad \text{a.e. } \; t \in [S_i,T_i],
\]
for all $i\ge i_1$. 
So, the sequence $((x_{\alpha_i})_e)_{i \in \N}$ is uniformly bounded in $W^{1,\infty}$ and also in $L^\infty$ say by a constant $R>0$. It follows from the Ascoli-Arzel\`a theorem that a subsequence of $((x_{\alpha_i})_e)_{i \in \N}$ (we do not relabel) converges uniformly to a function $\tilde x$. Moreover, the Banach-Alaoglu theorem guarantees that a subsequence of $((x_{\alpha_i})_e)_{i \in \N}$ weakly converges to a function in $W^{1,2}$. Thus $\tilde x$ is the weak limit in $W^{1,2}$ of a sequence uniformly bounded in $W^{1,\infty}$, and therefore $\tilde x \in W^{1,\infty}$ (with the same bound). It is easy to see that $\tilde x$ is constant outside $[S,T]$, so it is the extension of a function $x \in W^{1,{\infty}}([S, T];\R^n)$. 
Since for all $t \in [S,T]$, ${h}(t, x(t))=0$ we have that $x(t) \in C(t)$. We shall now prove that $x$ is the unique solution to $(SP1)$ on $[S,T]$.

As a consequence of (H4) and standard arguments (cf. \cite[Proposition 2.2]{BCG2025}) we also know that 
the sets $C(t)\cap R\B$ and $C_j(t) \cap R\B$, for $j=1,\dots,J$, are uniformly $r$-prox-regular for some $r>0$, for all $t \in [\bar{S}-\bar\delta,\bar{T}+\bar\delta]$. In view of \eqref{est1} and Lemmas \ref{kap} and \ref{phm} there exists $i_2\ge i_1$ such that $d_{C_j(t)}(x_{\alpha_i}(t))\le r$ and $d_{C(t)}(x_{\alpha_i}(t))\le r$ for all $t\in [S_i,T_i]$, for all $i\ge i_2$. 

\noindent
To simplify notation, the explicit time dependence of functions is often omitted in the following. Fix any $i\geq i_2$. For a.e.\ $t \in [S_i,T_i]$, and for all $y\in C(t)$ we have 
\allowdisplaybreaks[2]
\begin{align*}
    (f(t,  x_{\alpha_i}, u_i) - \dot{x}_{\alpha_i})& \cdot (y - x_{\alpha_i}) 
    = \frac{1}{\alpha_i} \underset{j=1}{\overset{J}{\sum}} h^+_j(t,x_{\alpha_i})^2 \, \nabla_x h_j(t,x_{\alpha_i}) \cdot (y - x_{\alpha_i}) \\
    &= \frac{1}{\alpha_i} \underset{j=1}{\overset{J}{\sum}} h^+_j(t,x_{\alpha_i})^2 \, \nabla_x h_j(t,x_{\alpha_i}) \cdot (y - \pi_{C_j(t)}(x_{\alpha_i})) \\
    &\quad +\frac{1}{\alpha_i} \underset{j=1}{\overset{J}{\sum}} h^+_j(t,x_{\alpha_i}(t))^2 \, \nabla_x h_j(t,x_{\alpha_i}(t)) \cdot (\pi_{C_j(t)}(x_{\alpha_i} ) - x_{\alpha_i} ) \\
    &= \frac{1}{\alpha_i} \underset{j=1}{\overset{J}{\sum}} h^+_j(t,x_{\alpha_i} )^2 \, \nabla_x h_j(t,\pi_{C_j(t)}(x_{\alpha_i} )) \cdot (y - \pi_{C_j(t)}(x_{\alpha_i} )) \\
    &\quad +\frac{1}{\alpha_i} \underset{j=1}{\overset{J}{\sum}} h^+_j(t,x_{\alpha_i} )^2 \, \left(\nabla_x h_j(t,x_{\alpha_i} )-\nabla_x h_j(t,\pi_{C_j(t)}(x_{\alpha_i} ))\right) \cdot (y - \pi_{C_j(t)}(x_{\alpha_i} )) \\
    &\quad +\frac{1}{\alpha_i} \underset{j=1}{\overset{J}{\sum}} h^+_j(t,x_{\alpha_i} )^2 \, \nabla_x h_j(t,x_{\alpha_i} ) \cdot (\pi_{C_j(t)}(x_{\alpha_i} ) - x_{\alpha_i} ). \\
\end{align*}
From the uniform $r$-prox-regularity of the sets $C(t)\cap R\B$ and $C_j(t)\cap R\B$, for $j=1, \dots,J$, we deduce that (for all $i\geq i_2$)
\begin{align*}
    (f(t, x_{\alpha_i}, u_i) - \dot{x}_{\alpha_i}) \cdot (y - x_{\alpha_i}) 
    &\leq A_1 \frac{K_{h}}{2r} |y - \pi_{C(t)}(x_{\alpha_i})|^2 \\
    &\quad +A_1 \,\underset{j=1}{\overset{J}{\sum}} \left|\nabla_x h_j(t,x_{\alpha_i})-\nabla_x h_j(t,\pi_{C_j(t)}(x_{\alpha_i}))\right| |y - \pi_{C_j(t)}(x_{\alpha_i})| \\
    &\quad +A_1 \,\underset{j=1}{\overset{J}{\sum}} \left|\nabla_x h_j(t,x_{\alpha_i})\right||\pi_{C_j(t)}(x_{\alpha_i}) - x_{\alpha_i}|. 
\end{align*}
The term on the right side converges to 
$
\frac{A_1K_h}{2r} |y - x(t)|^2 
$
as $i \to +\infty$.
The term on the left side is bounded by $A_1K_h(|y|+R)$, so it converges weakly (extracting a subsequence if necessary) in $L^2$ to $(f(., x(.), u(.)) - \dot{x}(.)) \cdot (y - x(.))$ as $i \to +\infty$.
It follows that 
\[
(f(t, x(t), u(t)) - \dot{x}) \cdot (y - x(t)) \leq \frac{A_1 K_{h}}{2r}|y - x(t)|^2, \text{ for all } y\in C(t), \text{ a.e. } t \in [S,T], 
\]
hence $(f(t, x(t), u(t)) - \dot{x}(t)) \in N_{C(t)}(x(t))$, and, so, $x$ is a solution to $(SP1)$. 
Invoking the uniform $r$-prox-regularity of $C(t)$, we can easily derive the uniqueness of the solution to $(SP1)$. Consequently, the limit does not depend on the choice of the subsequence of $((x_{\alpha_i})_e)_{i \in \N}$. This proves property~(ii) of Lemma~\ref{cvg}.

\end{proof}

\begin{rem}
Lemma~\ref{cvg} implies that, under assumptions {\rm (H2)}--{\rm (H5)}, the reference minimizer arc $\bar x$ of problem $(P)$ belongs to $W^{1,\infty}$.
\end{rem}

\begin{lem}\label{stg}
Suppose that (H2), (H3)* and {\rm (H4)} are satisfied. 
There exists $K>0$ such that the following holds. 
Let $(\alpha_i)_{i\in\N}$ be a decreasing sequence converging to $0$.  
Suppose that for each $i\in\N$ there exist processes
$([S_i,T_i],x_{\alpha_i},u_i)$ and $([S_i',T_i'],x_{\alpha_i}',u_i')$ for the problem $(CS_{\alpha_i})$ both verifying the conditions of Lemma \ref{cvg} so they approximate respectively some process $([S,T],x,u)$ and some process $([S',T'],x',u')$ for $(P)$.
Then there exists $i_0 \in \N$, such that, for all $i,j \geq i_0$,
\[
\|(x_{\alpha_i})_e-(x_{\alpha_j}')_e\|_{L^\infty(\bar S-\bar \delta,\bar T+ \bar \delta)}
\le
K\Big(
\alpha_i^{1/4}+\alpha_j^{1/4}
+|S_i-S_i'|+|T_i-T_i'|
+|x_{\alpha_i}(S_i)-x_{\alpha_j}'(S_j')|
+\sqrt{d_{\mathcal E}(u_i,u_j')}
\Big).
\]
Moreover, if for all $i\in\N$ one has a process $([\bar S, \bar T],x_{\alpha_i},\bar u)$ for $(CS_{\alpha_i})$, then
\[
\|x_{\alpha_i}-\bar x\|_{L^\infty(\bar S,\bar T)} \le K\,\alpha_i^{1/4}.
\]
\end{lem}

\begin{proof}
Denote for all $k \in \{1,\dots,J\}$, all $i\in \N$, all $t\in [S_i,T_i]$ and all $t'\in [S_i',T_i']$,
\[
C_k^{i}(t) := \left\{x \in \R^n :\, h_k(t,x) \leq h^+_k(t,x_{\alpha_i}(t))\right\}, \qquad 
\Tilde{C}_k^{i}(t') := \left\{x \in \R^n :\, h_k(t',x) \leq h^+_k(t',x_{\alpha_i}'(t'))\right\},
\]
and for all $k \in \{1,\dots,J\}$, all $i,j\in \N$ and all $t\in [S_i \vee S_j',T_i\wedge T_j']$,
\[
\pi_k^{i,j}(t):= \pi_{C_k^{i}(t)}(x_{\alpha_j}'(t)), \qquad 
\Tilde{\pi}_k^{j,i}(t):= \pi_{\Tilde{C}_k^{j}(t)}(x_{\alpha_i}(t)),
\]
note that $C_k(t) \subset C_k^{i}(t) \ni x_{\alpha_i}(t)$ and $C_k(t) \subset \Tilde{C}_k^{j}(t) \ni x_{\alpha_j}'(t)$.

According to Lemma \ref{cvg}, there exist $A>0$ (not depending on the sequences of approximating processes) and $i_0 \in \N$ such that, for all $i \geq i_0$, 
    \[
    \sum_{k=1}^{J} h_k(t,x_{\alpha_i}(t))^3 \leq 3A\alpha_i^{3/2}, \qquad \sum_{k=1}^{J} h_k(t,x_{\alpha_i}'(t))^3 \leq 3A{\alpha_i}^{3/2},
    \]
for all $ t\in [S_i \vee S_i',T_i\wedge T_i']$, and such that
    \[
    \|(x_{\alpha_i})_e - x_e\|_{L^\infty} \leq r', \qquad\|(x_{\alpha_i}')_e -  x'_e\|_{L^\infty} \leq r',
    \]
where $r' \in \left(0,r\wedge \frac{\bar \delta}{4}\right)$ is such that $C_k^{i}(t)$ and $\Tilde{C}_k^{i}(t) $ are $r'$-prox-regular.
Hence, the projections defined above are singletons.

Let $i,j \geq i_0$ and consider the function $z_{i,j}(t):=  \frac{1}{2}|x_{\alpha_i}(t) -x_{\alpha_j}'(t)|^2$. For all $t\in [S_i \vee S_j',T_i\wedge T_j']$,
\begin{align*}
    \dot{z}_{i,j}(t) 
    &= (\dot{x}_{\alpha_i}(t) -\dot{x}_{\alpha_j}'(t))\cdot(x_{\alpha_i}(t) -x_{\alpha_j}'(t)) \\
    &= (f(t,x_{\alpha_i}(t),u_i(t)) -f(t,x_{\alpha_j}'(t),u_j'(t)))\cdot(x_{\alpha_i}(t) -x_{\alpha_j}'(t)) \\
    & \quad - \sum_{k=1}^J \left( \frac{h^+_k(t,x_{\alpha_i}(t))^2}{\alpha_i}\nabla_x h_k(t,x_{\alpha_i}(t)) - \frac{h^+_k(t,x_{\alpha_j}'(t))^2}{\alpha_j}\nabla_x h_k(t,x_{\alpha_j}'(t))\right)\cdot(x_{\alpha_i}(t) -x_{\alpha_j}'(t)) \\
    &= (f(t,x_{\alpha_i}(t),u_i(t)) -f(t,x_{\alpha_j}'(t),u_i(t)))\cdot(x_{\alpha_i}(t) -x_{\alpha_j}'(t)) \\ 
    & \quad +(f(t,x_{\alpha_j}'(t),u_i(t)) -f(t,x_{\alpha_j}'(t),u_j'(t)))\cdot(x_{\alpha_i}(t) -x_{\alpha_j}'(t)) \\
    & \quad + \sum_{k=1}^J \frac{h^+_k(t,x_{\alpha_i}(t))^2}{\alpha_i}\nabla_x h_k(t,x_{\alpha_i}(t))\cdot( x_{\alpha_j}'(t) - x_{\alpha_i}(t)) \\
    & \quad + \sum_{k=1}^J \frac{h^+_k(t,x_{\alpha_j}'(t))^2}{\alpha_j}\nabla_x h_k(t,x_{\alpha_j}'(t))\cdot(x_{\alpha_i}(t) -x_{\alpha_j}'(t))\\
\end{align*}
From assumption (H3)* we obtain
\[
|f(t,x_{\alpha_i}(t),u_i(t)) -f(t,x_{\alpha_j}'(t),u_i(t))| \leq k_f(t) |x_{\alpha_i}(t) -x_{\alpha_j}'(t)|,
\]
\[
|f(t,x_{\alpha_j}'(t),u_i(t)) -f(t,x_{\alpha_j}'(t),u_j'(t))| \leq c_f m(u_i(t),u_j'(t)). 
\]
where 
\[
m(u_i(t),u_j'(t)) :=
\begin{cases}
1 & \text{if } u_i(t)\neq u_j'(t),\\
0 & \text{if } u_i(t)= u_j'(t).
\end{cases}
\]
It follows that
\begin{align*}
    \nabla_x h_k(t,x_{\alpha_i}(t))\cdot & ( x_{\alpha_j}'(t) - x_{\alpha_i}(t)) 
    = \nabla_x h_k(t,x_{\alpha_i}(t))\cdot( x_{\alpha_j}'(t) - \pi_k^{i,j}(t))\\
		& \hspace{3 in}+ \nabla_x h_k(t,x_{\alpha_i}(t))\cdot(\pi_k^{i,j}(t) -  x_{\alpha_i}(t)) \\
    &\leq |\nabla_x h_k(t,x_{\alpha_i}(t))| |x_{\alpha_j}'(t) - \pi_k^{i,j}(t)| + \frac{|\nabla_x h_k(t,x_{\alpha_i}(t))|}{2r'} |\pi_k^{i,j}(t) -  x_{\alpha_i}(t) |^2 \\
    &\leq K_h\left( d_{C_k(t)}(x_{\alpha_j}'(t) ) + \frac{2}{r'} |x_{\alpha_j}'(t) -  x_{\alpha_i}(t) |^2 \right) \\
    &\leq K_h \left((3A)^{1/3}\Lambda \sqrt{\alpha_j}+ \frac{2}{r'} |x_{\alpha_i}(t) - x_{\alpha_j}'(t) |^2 \right).
\end{align*}
Similarly
\[
\nabla_x h_k(t,x_{\alpha_j}'(t))\cdot(x_{\alpha_i}(t) -x_{\alpha_j}'(t)) \leq K_h \left((3A)^{1/3}\Lambda \sqrt{\alpha_i}+ \frac{2}{r'} |x_{\alpha_i}(t) - x_{\alpha_j}'(t) |^2 \right).
\]
Let $C_1(t) := k_f(t) +  (3A)^{2/3}k_2 K_h \frac{4}{r'} $ and $C_2:=3Ak_2 K_h\Lambda$, so that
\[
\dot{z}_{i,j}(t) \leq C_1(t) z(t) + c_f m(u_i,u_j')\sqrt{z(t)} + C_2 (\sqrt{\alpha_i} + \sqrt{\alpha_j}).
\]
Then, by Gronwall's lemma,
\[
z_{i,j}(t) \leq \left(z(S_i \vee S_j')+C_2 (t-S_i \vee S_j')(\sqrt{\alpha_i}+\sqrt{\alpha_j}) + c_f d_\mathcal{E}(u_i,u_j') \right)e^{(t-S_i \vee S_j')c_f+ \int_{S_i \vee S_j'}^t C_1(s) \dd s}.
\]
Moreover,
\[
|x_{\alpha_i}(S_i \vee S_j') -x_{\alpha_j}'(S_i \vee S_j')| \leq |x_{\alpha_i}(S_i) -x_{\alpha_j}'(S_j')| + (c_f+ (3A)^{2/3}k_2K_h)|S_i - S_j'|.
\]
Define $K :=\sqrt{(1+(c_f+ (3A)^{2/3}k_2K_h)^2+C_2(\bar{T}-\bar{S}+2\delta) + c_f)e^{(\bar{T}-\bar{S}+2\delta)c_f+\|C_1\|_{L^1}}}$. Thus, for all \(t\in[S_i\vee S_j',\,T_i\wedge T_j']\),
\[
|x_{\alpha_i}(t) -x_{\alpha_j}'(t)| \leq K \left({\alpha_i}^{1/4} + {\alpha_j}^{1/4} +|S_i - S_j'|+|x_{\alpha_i}(S_i) -x_{\alpha_j}'(S_j')| +  \sqrt{d_\mathcal{E}(u_i,u_j')} \right).
\]
Furthermore, for $t \geq T_i\wedge T_j'$
\begin{align*}
    |(x_{\alpha_i})_e(t) -(x_{\alpha_j}')_e(t)| 
    &\leq |(x_{\alpha_i})_e(t)-x_{\alpha_i}(T_i) | + |x_{\alpha_i}(T_i)-x_{\alpha_j}'(T_j')| + |x_{\alpha_j}'(T_j') -(x_{\alpha_j}')_e(t)| \\
    &\leq |x_{\alpha_i}(T_i)-x_{\alpha_j}'(T_j')|+ (c_f+ (3A)^{2/3}k_2K_h)|T_i-T_j'|,
\end{align*}
Similarly, for all $t \leq S_i\vee S_j'$, we obtain
\[
|(x_{\alpha_i})_e(t) -(x_{\alpha_j}')_e(t)| \leq |x_{\alpha_i}(S_i) -x_{\alpha_j}'(S_j')| + (c_f+ (3A)^{2/3}k_2K_h)|S_i-S_j'|.
\]
Hence, for all $t \in [\bar{S}-\delta,\bar{T}+\delta]$,
\[
|(x_{\alpha_i})_e(t) -(x_{\alpha_j}')_e(t)|
\leq K\Big(
\alpha_i^{1/4}+\alpha_j^{1/4}
+|S_i-S_i'|+|T_i-T_i'|
+|x_{\alpha_i}(S_i)-x_{\alpha_j}'(S_j')|
+\sqrt{d_{\mathcal E}((u_i)_e,(u_j')_e)}
\Big).
\]
\end{proof}

\section{Proof of the Free End-Time Maximum Principle (Measurable Time-dependent Case)}\label{Ek-sec}

\noindent
{\bf Proof of Theorem \ref{pm}.} 
We start constructing a family of auxiliary dynamic optimization problems. Set 
\[
X:=\left\{ \left( [S,T],y,u \right) \text{ process for } (SP) \text{ such that } 
|S-\bar S|+|T-\bar T|+|y(S)-\bar x(\bar S)| \leq \bar \delta/2 \right\}.
\]
We endow $X$ with the metric \( d_X : X \times X \to [0,+\infty) \) defined by
\[
d_X\left( ([S,T], y, u ), ([S', T'], y', u') \right) = |S - S'| + |T - T'| + |y(S) - y'(S')| + d_{\mathcal{E}}(u, u').
\]
It is straightforward to see that $(X,d_X)$ is a complete metric space. Let $(\varepsilon_i)_{i \in \N}$ be a decreasing sequence converging to $0$. For a fixed $i \in \N$, consider the optimization problem
\begin{equation*}
(SP_{i}) \left\{ \begin{array}{ll}
\text{Minimize} &
\Psi_{i}(([S,T],y,u)) := \big(g(S,y(S),T,y(T)) - g(\Bar{S},\Bar{x}(\Bar{S}),\Bar{T},\Bar{x}(\Bar{T})) + \varepsilon_i^2 \big) \\
& \qquad \vee d_E(S,y(S),T,y(T)) \\
\text{over processes} & \left( [S,T],y,u \right) \in X.
 \\
\end{array}
\right.
\end{equation*}
Observe that $([\bar{S},\bar{T}],\bar x,\bar{u})$ is an $\varepsilon_i^2$-minimizer for $(SP_i)$ and $\Psi_{i}$ is continuous on $X$. Therefore Ekeland's variational principle asserts the existence of $([\sigma_i,\tau_i],y_i,v_i) \in X$ which is a minimizer for the problem
\begin{equation*}
(Q_{i}) \left\{ \begin{array}{ll}
\text{Minimize} &
\Phi_{i}(([S,T],y,u)) := \big(g(S,y(S),T,y(T)) - g(\Bar{S},\Bar{x}(\Bar{S}),\Bar{T},\Bar{x}(\Bar{T})) + \varepsilon_i^2 \big) \\
& \qquad \vee d_E(S,y(S),T,y(T)) + \varepsilon_id_X(([S,T],y,u),([\sigma_i,\tau_i],y_i,v_i)) \\
\text{over processes} & \left( [S,T],y,u \right) \in X,
 \\
\end{array}
\right.
\end{equation*}
and satisfies $d_X(([\sigma_i,\tau_i],y_i,v_i),([\bar S,\bar T],\bar x,\bar u)) \leq \varepsilon_i$. 
Observe that 
$$
\Psi_{i}(([S,T],y,u)) >0, \quad \text{for all } \; ( [S,T],y,u ) \in X,
$$
since, otherwise, we would arrive at a contradiction of the minimality of $([\bar S,\bar T],\bar x,\bar u)$. 
As a consequence, we obtain that  
\begin{equation*}
\min_X \Phi_{i}(( [S,T],y,u )) =    \Phi_{i}(([\sigma_i,\tau_i],y_i,v_i)) >0,
\end{equation*}
Then, from Lemma \ref{cvg} and Lemma \ref{stg}, it follows that, for each $i\ge 1$, we can find $\rho_i \in (0,\varepsilon_i/2)$ and $\alpha_i>0$ such that 
\begin{equation}\label{rha}
   \Phi_{i}(( [S,T],x,u )) \ge \Phi_{i}([\sigma_i,\tau_i],y_i,v_i) - \frac{\rho_i^2}{2}>0, 
\end{equation}
and 
\begin{equation}\label{rhb}
   | \Phi_{i}( ([S,T],x,u )) - \Phi_{i}( ([S,T],y,u ))| < \frac{\rho_i^2}{2},
\end{equation}
for all $\left( [S,T],x,u \right) \in W_i$, where $\left( [S,T],y,u \right) \in X$ is the corresponding element taking the same time interval $[S,T]$ and the same control $u$, and
\[
\begin{aligned}
W_i:=\Big\{ \left( [S,T],x,u \right) & \text{ process for } (CS_{\alpha_i}) \text{ such that } x(S) \in C(S) \text{ and } \\
&|S-\bar S|+|T-\bar T|+|x(S)-\bar x(\bar S)| \leq \bar \delta/2  \Big\}.
\end{aligned}
\]
Write $\tilde x_i \in W^{1,1}([\sigma_i,\tau_i];\R^n)$ the solution to $(CS_{\alpha_i})$ associated with the control $v_i$ such that $\tilde x_i(\sigma_i)=y(\sigma_i)$.
We obtain that a process $([\sigma_i,\tau_i],\tilde x_i,v_i)$ for $(CS_{\alpha_i})$. Notice that $([\sigma_i,\tau_i],\tilde x_i,v_i)\in W_i$ is a $\rho_i^2$-minimizer for the continuous map $\Phi_i : (W_i,d_X) \to (\mathbb{R},|\cdot|)$. 
It is straightforward to see that $(W_i,d_X)$ is a complete metric space.

Applying Ekeland's variational principle again, for each $i\ge 1$, we obtain a process $([S_i,T_i],x_i,u_i) \in W_i$ such that
\begin{equation}
d_X\big(([S_i,T_i],x_i,u_i),([\sigma_i,\tau_i],\tilde x_i,v_i)\big) \le \rho_i,
\end{equation}
and
\begin{equation}
\Phi_i([S,T],x,u) + \rho_i\, d_X\big(([S,T],x,u),([S_i,T_i],x_i,u_i)\big)
\ge 
\Phi_i([S_i,T_i],x_i,u_i)
\end{equation}
for all $([S,T],x,u) \in W_i$. Therefore, the process $([S_i,T_i],x_i,u_i)$ is a local minimizer for the dynamic optimization problem 
\begin{equation*}
(P_i) \left\{ \begin{array}{ll}
\text{Minimize } &
 \left(g(S,x(S),T,x(T)) - g(\Bar{S},\Bar{x}(\Bar{S}),\Bar{T},\Bar{x}(\Bar{T})) +\varepsilon_i^2 \right) \vee d_E(S,x(S),T,x(T)) \\
&+ \varepsilon_i \left(|S-\sigma_i|+|T-\tau_i|+|x(S)-y_i(\sigma_i)|+z(T) \right) \\&
+ \rho_i \left(|S-S_i|+|T-T_i|+|x(S)-x_i(S_i)|+w(T) \right) \\
\text{over} & 
\text{intervals }[S,T] \subset \R, \text{ arcs } (x,z,w) \in W^{1,1}([S,T];\mathbb{R}^{n}\times \R\times \R), \\
& \text{and measurable functions } u : [S,T] \to \R^{m}, \\
\text{such that} & 
\dot{x}(t) = f(t,x(t),u(t)) -\frac{1}{\alpha_i} \nabla_x h(t, x(t)) \text{  for a.e. } t \in [S,T], \\
& \dot{z}(t) = m_i(t,u(t)) \text{  for a.e. } t \in [S,T], \\
& \dot{w}(t) = m_i'(t,u(t)) \text{  for a.e. } t \in [S,T], \\
& u(t) \in U(t) \text{  for a.e. } t \in [S,T],\\
& z(S)=0, \, w(S)=0, \, \\
& h_j(S,x(S))\le 0 \text{  for all } j=1,\dots,J ,\\
& |S-\bar S|+|T-\bar T|+|x(S)-\bar x(\bar S)| \leq \bar \delta/2, 
\end{array}
\right.
\end{equation*}
where  
\(
m_i(t,u):=
\begin{cases}
\,1 & \text{if } u\neq v_i(t),\\
\,0 & \text{otherwise},
\end{cases}
\)
and 
\(
m_i'(t,u):=
\begin{cases}
\,1 & \text{if } u\neq u_i(t),\\
\,0 & \text{otherwise}.
\end{cases}
\)

 \ \\

\noindent
{\bf Analysis of multipliers and their convergence properties.} 

\noindent
We are now ready to apply the necessary conditions of optimality for the local minimizer $([S_i,T_i],x_i,u_i)$ of the dynamic optimization problem $(P_i)$. We observe that for all $i$ sufficiently large the constraint `$|S-\bar S|+|T-\bar T|+|x(S)-\bar x(\bar S)| \leq \bar \delta/2$' is never active. Moreover, there exists $R>0$ such that $||x_i||_{L^\infty}< R$ for all $i\ge 1$. 
The free end-time maximum principle (see \cite[Theorem~9.6.2]{Bettiol-Vinter-book2024} or \cite[Theorem~3.1]{Bettiol-Vinter-SICON2025})  
yields the existence of multipliers $p_i \in W^{1,1}([S_i,T_i];\R^n)$, $\lambda_i^{(1)}, \ \lambda_i^{(2)}, \ \lambda_i^{(3)} \in [0,1]$ and $\beta^i_j\ge 0$, with $\lambda_i^{(1)} + \lambda_i^{(2)} + \lambda_i^{(3)} =1$ and $\sum_{j=1}^J \beta^i_j=1$, satisfying the following properties (for $i$ large enough).
For a.e. $t\in [S_i,T_i]$,
\begin{equation}\label{eq:ii-vr}
\begin{aligned}
-\dot{p}_i(t) \in\;& 
\operatorname{co}\,\partial_x \mathcal{H}(t,x_i(t),p_i(t),u_i(t)) \\
&-\frac{1}{\alpha_i}\sum_{j=1}^J h_j^+(t,x_i(t))^2 \, p_i(t)\,D^2_x h_j(t,x_i(t)) \\
&-\frac{2}{\alpha_i}\,\sum_{j=1}^J h_j^+(t,x_i(t))\, p_i(t)\cdot \nabla_x h_j(t,x_i(t)) \, 
\nabla_x h_j(t,x_i(t)).
\end{aligned}
\end{equation}
Moreover, there exist $\xi_0^i$ and $\xi_1^i$ such that
\begin{equation}\label{eq:iii-vr}
\begin{aligned}
(-\xi_0^i,p_i(S_i),\xi_1^i,-p_i(T_i))
\in\;&
\lambda_i^{(1)}\,\partial g(S_i,x_i(S_i),T_i,x_i(T_i)) \\
&+ \lambda_i^{(2)}\,\partial d_E(S_i,x_i(S_i),T_i,x_i(T_i)) \\
&+ \lambda_i^{(3)}\,\sum_{j=1}^J \beta^i_j\nabla h_j(S_i,x_i(S_i))\times\{0\}\times\{0\}\\
&+ 2\varepsilon_i \,\big([-1,1]\times\mathbb{B}\times\{0\}\times\{0\}\big)
\end{aligned}
\end{equation}
where
\begin{align}
\xi_0^i &\in 
\underset{t\to S_i}{sub\text{-}ess}\!\left[
\sup_{u\in U(t)} p_i(S_i)\cdot f(t,x_i(S_i),u)
\right] + 2\varepsilon_i [-1,1]
-\frac{p_i(S_i)}{\alpha_i}\cdot \nabla_x h(S_i,x_i(S_i)), \label{t1}\\
\xi_1^i &\in 
\underset{t\to T_i}{super\text{-}ess}\!\left[
\sup_{u\in U(t)} p_i(T_i)\cdot f(t,x_i(T_i),u)
\right] + 2\varepsilon_i [-1,1]
-\frac{p_i(T_i)}{\alpha_i}\cdot \nabla_x h(T_i,x_i(T_i)).\label{t2}
\end{align}
Moreover, we have 
\begin{equation}\label{eq:iv-vr}
p_i(t)\cdot f(t,x_i(t),u_i(t))
\ge
p_i(t)\cdot f(t,x_i(t),u) - 2\varepsilon_i,
\quad \text{for all } u\in U(t), \; a.e. \; t\in[S_i,T_i].
\end{equation}

\noindent
Denote $K_2:=\underset{j \in \{1,\dots, J\}}{\underset{(s,x) \in [\bar S - \bar \delta, \bar T + \bar \delta]\times R\B}{\max}} |D^2_x h_j(s,x)|$.
Since, for a.e. $t \in [S_i,T_i]$
\begin{align*}
    p_i(t) \cdot D^2_x h(t,x_i(t)) p_i(t)
    &= 
    \sum_{j=1}^J h_j^+(t,x_i(t))^2 p_i(t) \cdot D^2_x h_j(t,x_i(t)) p_i(t) + 2 h_j^+(t,x_i(t)) |p_i(t) \cdot \nabla_x h_j(t,x_i(t))|^2 \\
    &\geq - (3A)^{2/3}k_2 \alpha_i
    |p_i(t)|^2 K_2,
\end{align*}
we obtain that
\begin{align*}
    -\frac{\dd}{\dd t}|p_i(t)| 
    &\leq |p_i(t)| k_f(t) - \frac{1}{\alpha_i |p_i(t)|} p_i(t) \cdot D^2_x h(t,x_i(t)) p_i(t) \\
    &\leq |p_i(t)| \left(k_f(t) + (3A)^{2/3}k_2 K_2\right), \quad a.e. \; t \in [S_i,T_i].
\end{align*}
Write $q_i: s \mapsto p_i(T_i-s)$, then for a.e. $t \in [0,T_i-S_i]$,
\[
\frac{\dd}{\dd t}|q_i(t)| \leq |q_i(t)| \left(k_f(T_i-t) + (3A)^{2/3}k_2 K_2\right).
\]
Thus Gronwall's lemma ensure that the sequence $(q_i)_{i \in \N}$ is bounded in $L^\infty$. Then, the sequence $(p_i)_{i \in \N}$ is also bounded in $L^\infty$.\\

\noindent
For a.e. $t \in [S_i,T_i]$, we have
\begin{align*}
    \dot{p}_i(t) 
    &\in 
    \frac{2}{\alpha_i} \underset{j =1}{\overset{J}{\sum}} h^+_j(t,x_i(t)) p_i(t) \cdot \nabla_x h_j(t,x_i(t)) \, \nabla_x h_j(t,x_i(t)) + B_i(t),
\end{align*}
where
\[
B_i(t) := \frac{p_i(t)}{\alpha_i} \underset{j =1}{\overset{J}{\sum}} h^+_j(t,x_i(t))^2 \, D^2_x h_j(t,x_i(t)) - {\rm co \,} \partial_x p_i(t) \cdot f(t,x_i(t),u_i(t)).
\]
Note that the following inclusion holds
\[
B_i(t) \subset \|p_i\|_{L^\infty(S_i,T_i)}(k_f(t) +(3A)^{2/3}k_2 K_2)\B.
\]
For all $j \in \{1, \dots, J\}$, we set
\[
\chi_i^j(t) := p_i(t)\cdot \nabla_x h_j(t,x_i(t)), \qquad \text{for all } t \in [S_i,T_i].
\]
Then, for a.e. $t \in [S_i,T_i]$,
\begin{align*}
    \dot{\chi}_i^j(t) &= \dot{p}_i(t) \cdot \nabla_x h_j(t,x_i(t)) + p_i(t) \cdot D^2_x h_j(t,x_i(t)) \dot{x}_i(t) + p_i(t) \cdot \nabla_t \nabla_x h_j(t,x_i(t)) \\
    &\in \frac{2}{\alpha_i} \underset{k \in I_\eta(t,x_i(t))}{{\sum}} h^+_k(t,x_i(t)) {\chi}_i^k(t) \, \nabla_x h_k(t,x_i(t)) \cdot \nabla_x h_j(t,x_i(t)) \\
    & \quad + B_i(t) \cdot \nabla_x h_j(t,x_i(t)) + p_i(t) \cdot D^2_x h_j(t,x_i(t)) \dot{x}_i(t) + p_i(t) \cdot 
		\nabla_t\nabla_x h_j(t,x_i(t)) .
\end{align*}
Set $\widehat K:=\underset{j \in \{1,\dots, J\}}{\underset{(s,x) \in [\bar S - \bar \delta, \bar T + \bar \delta]\times R\B}{\max}} |
\nabla_t\nabla_x h_j(s,x)|$ and
\[
D^j_i(t) := B_i(t) \cdot \nabla_x h_j(t,x_i(t)) + p_i (t) \cdot D^2_x h_j(t,x_i(t)) \dot{x}_i(t) + p_i (t) \cdot 
\partial_t \nabla_x h_j(t,x_i(t)).
\]
Note that
\[
D^j_i(t)\subset \|p_i\|_{L^\infty}\left(\big(k_f(t) +(3A)^{2/3}k_2 K_2\big)K_h + K_2 \big(c_f+(3A)^{2/3}k_2K_h\big) + \widehat K \right)\B.
\]
so, from the $L^\infty$-boundedness of $(p_i)_{i \in \N}$ we deduce that there exists a positive function $c \in L^1$ such that $D^j_i(t)\subset c(t)\B$.

On the other hand, for a.e. $t \in [S_i,T_i]$, (to simplify notation, we may omit the explicit time dependence of functions in the following)
\begin{align*}
    \sum_{j = 1}^J \dot{\chi}^j_i \operatorname{sgn}({\chi}^j_i) &\in \frac{2}{\alpha_i} \sum_{j = 1}^J \left(\underset{k \in I_\eta^+(t,x_i)}{{\sum}} h^+_k(t,x_i) {\chi}_i^k \, \nabla_x h_k(t,x_i) \cdot \nabla_x h_j(t,x_i) \right) \operatorname{sgn}({\chi}^j_i) +\sum_{j = 1}^J D^j_i \operatorname{sgn}({\chi}^j_i) \\
    &= \frac{2}{\alpha_i} \underset{k \in I_\eta^+(t,x_i)}{\sum} \left(
		\sum_{j = 1}^J\operatorname{sgn}({\chi}^k_i)  \, \nabla_x h_k(t,x_i) \cdot \nabla_x h_j(t,x_i)\operatorname{sgn}({\chi}^j_i)  \right) |\chi_i^k| h^+_k(t,x_i)\\
    & \quad +\sum_{j = 1}^J D^j_i \operatorname{sgn}({\chi}^j_i) \\
    &= \frac{2}{\alpha_i} \underset{k \in I_\eta^+(t,x_i)}{\sum} 
		\Big(\underset{j\neq k}{\underset{j \in I_\eta(t,x_i)}{{\sum}}} \operatorname{sgn}({\chi}^k_i)  \, \nabla_x h_k(t,x_i) \cdot \nabla_x h_j(t,x_i)\operatorname{sgn}({\chi}^j_i) \\
		&\hspace{2.5 in}	+ \left|\nabla_x h_k(t,x_i)\right|^2 \Big) |\chi_i^k| h^+_k(t,x_i)\\
    & \quad +\sum_{j = 1}^J D^j_i \operatorname{sgn}({\chi}^j_i). \\
\end{align*}
Applying (H5)$^*$ for a.e. $t \in [S_i,T_i]$ and for each $k \in I_\eta^+(t,x_i(t))$, we obtain 
\[
\sum_{j = 1}^J \dot{\chi}^j_i \operatorname{sgn}({\chi}^j_i) \geq \frac{2(1 - \rho)}{\alpha_i} 
|\nabla_x h_k(t,x_i(t))|^2 \sum_{j = 1}^J |\chi_i^j| h^+_j(t,x_i) -Jc(t).
\]
It follows that, for a.e. $t \in [S_i,T_i]$,
\begin{equation}\label{brn}
    \frac{1}{\alpha_i} \overset{J}{\underset{j=1}{\sum}}  |\chi_i^j| h^+_j(t,x_i) \leq \left( {2(1 - \rho){\underset{k \in I_\eta^+(\tau,x)}{\underset{(\tau,x) \in [\bar S - \bar \delta, \bar T + \bar \delta]\times R\B}{\inf}} |\nabla_x h_k(\tau,x)|^2}}\right)^{-1} \left(\left({\sum_{j=1}^J}\dot{\chi}^j_i \operatorname{sgn}({\chi}^j_i) \right)+Jc\right).
\end{equation}
Notice that
\begin{align*}
    \int_{S_i}^{T_i} \sum_{j=1}^J \dot{\chi}^j_i(t) \operatorname{sgn}({\chi}^j_i(t)) \, \dd t
    &= \sum_{j=1}^J |{\chi}^j_i(T_i)| - |{\chi}^j_i(S_i)| \\
    &\leq J\|p_i\|_{L^\infty} K_h.
\end{align*}
Hence, all the terms on the right side of the inclusion \eqref{eq:ii-vr} are uniformly bounded in $L^1$; consequently, the sequence $(\dot{p}_i)_{i \in \N}$ is uniformly bounded in $L^1$.\\

\noindent
By Helly's theorem, we may extract a subsequence of $(p_i)_{i\in\N}$ (we do not relabel) that converges pointwise on $[\bar S-\bar\delta,\bar T+\bar\delta]$ to a function $p$ of bounded variation, and such that the associated measures $(\dot p_i(t)\,\dd t)$ converge weakly$^*$ in $C^0([\bar S-\bar\delta,\bar T+\bar\delta];\R^n)^*$ to the (signed) Radon measure $\dd p$ (here, we consider the extension of $p_i$ and $\dot p_i$ on the compact time interval $[\bar S-\bar\delta,\bar T+\bar\delta]$ according to the preceding convention, by constant extrapolation for arcs and by zero extrapolations for their derivatives). 
If necessary, we modify the values of $p$ at $\bar S$ and $\bar T$ so that $p(\bar S - \bar \delta) = p(\bar{S})$ and $p(\bar T+\bar\delta) = p(\bar{T})$.

Along the same subsequence we may also assume that
$\lambda_i^1\to\lambda\in[0,1]$, 
$\xi_0^i\to\xi_0\in\R$, 
$\xi_1^i\to\xi_1\in\R$, 
and, by the Banach--Alaoglu theorem, we may extract a further subsequence (we do not relabel) such that for each $j \in \{1,\dots,J\}$, the sequence of functions $\displaystyle \left(\frac{h_j^+\!\bigl(t,x_i\bigr)^2}{\alpha_i}\right)_{i \in \N}$ extended on $[\bar S-\bar\delta,\bar T+\bar\delta]$ converges weakly in $L^2(\bar S-\bar\delta,\bar T+\bar\delta)$ to a function $\eta_j$ that takes values in $[0,(3A)^{2/3}k_2]$.
It follows that condition (ii) of Theorem \ref{pm} holds.

Moreover, the right term of inequality \eqref{brn} is bounded in $L^1$. Hence, for each $j \in \{1,\dots,J\}$, from the sequence of measures extended on $[\bar S - \delta,\bar T+\bar\delta]$
\[
\left(\frac{2}{\alpha_i}h_j^+(t,x_i(t))p_i(t)\cdot \nabla_xh_j(t,x_i(t)) \, \dd t \right)_{i \in \N}
\]
we can extract a subsequence (we do not relabel) that converges weakly$^*$ in $C^0([\bar S - \delta,\bar T+\bar\delta];\R^n)$ to Radon measure $\dd \mu_j$.  
Combined with the definition of $(\eta_j)_{1\leq j\leq J}$, this yields conditions~(iii) and (vii) of Theorem \ref{pm}. Since, for each $j \in \{1,\dots,J\}$,
\[
\left(
\frac{2}{\alpha_i}
h_j^+\!\bigl(t,x_i(t)\bigr)\,
\bigl|p_i(t)\cdot\nabla_x h_j(t,x_i(t))\bigr|^2
\,\dd t
\right)_{i\in\N}
\]
is a sequence of positive measures, passing to the limit gives condition (viii) of Theorem \ref{pm}. Also, for each $j \in \{1,\dots,J\}$
\[
\eta_j(t)\,p(t)\cdot\nabla_x h_j(t,\bar x(t)) \dd t= h_j^+ (t,\bar x(t)) \dd \mu_j (t) = 0.
\]
from which condition (vi) of Theorem \ref{pm} follows.

\noindent
We now examine the transversality condition. 
By definition of $(\eta_j)_{1\leq j\leq J}$ we obtain:
\[
    \frac{\nabla_x h(S_i,x_i(S_i))}{\alpha_i} \underset{i \to + \infty}{\longrightarrow} \sum_{j=1}^J \eta_j(\bar S) \nabla_x h_j(\bar S,\bar x(\bar S)) \in N_{C(\bar S)}(\bar x (\bar S)),
\]
\[
    \frac{\nabla_x h(T_i,x_i(T_i))}{\alpha_i} \underset{i \to + \infty}{\longrightarrow} \sum_{j=1}^J \eta_j(\bar T) \nabla_x h_j(\bar T,\bar x(\bar T)) \in N_{C(\bar T)}(\bar x (\bar T)).
\]
Then, invoking the limiting and stability properties of the sub- and super-essential values (see \cite{Bettiol-Vinter-SICON2025} and \cite[Chapter 9]{Bettiol-Vinter-book2024}), from \eqref{eq:iii-vr}, \eqref{t1} and \eqref{t2}, passing to the limit as $i \to \infty$, we arrive at the transversality condition (iv).
Condition (v) of Theorem \ref{pm} follows from \eqref{eq:iv-vr}.
Finally, condition (i) of Theorem \ref{pm} follows from \eqref{rha} with a standard argument (cf. \cite{Bettiol-Vinter-book2024}). \hfill $\square$

\paragraph{Proof of Proposition \ref{libre}.}
In addition to (H1)-(H5), we assume that $E=E_0\times\R^n$, where $E_0$ is a closed subset of ${\rm Gr}\, C \times\R$.
In this case, applying the necessary conditions to the dynamic optimization problem $(P_i)$, in place of \eqref{eq:iii-vr} we obtain the following (simpler) transversality conditions for the local minimizer $([S_i,T_i],x_i,u_i)$: 
\begin{align*}
(-\xi_0^i,p_i(S_i),\xi_1^i,-p_i(T_i))
    \in & \lambda_i\partial g(S_i,x_i(S_i),T_i, x_i(T_i)) + (1-\lambda_i)\,\partial d_{E_0}(S_i,x_i(S_i),T_i) \times \{0\} \\
    & + 2\varepsilon_i\bigl([-1,1]\times\B\times \{0\} \times\{0\}\bigr),
\end{align*}
where $\lambda_i\in [0,1]$. Using standard arguments, it is immediate to see that, passing to the limit as $i \to \infty$, we necessarily have $\lambda_i \to \lambda >0$. Normalizing, we arrive at the desired conclusion. 

\hfill $\square$

\ \\

\paragraph{Proof of Theorem \ref{pmo}}

Observe that there exists a constant $R>0$ such that $\|\bar x\|_{L^\infty([\bar S,\bar T];\mathbb{R}^n)} + 2\bar \delta \leq R$ and (reducing the size of $r>0$ if necessary) $C_j(t) \cap R\B$, for all $j \in \{1,\dots,J\}$, $C(t) \cap \left(\tilde y(t)+ \delta_0 \mathbb{B} \right)$ and $C(t) \cap R\B$ are uniformly $r$-prox-regular for all $t \in [\bar S- \bar \delta, \bar T +\bar \delta]$. 
We define the scalar function $\tilde h : \mathbb{R} \times \mathbb{R}^n \to \mathbb{R}$ by
\[ 
\tilde h(t,x) := \frac{1}{3}\sum_{j=1}^{J+1} h_j^+(t,x)^3,
\]
where the family of functions $(h_j)_{1\leq j \leq J}$ is provided by $(\widetilde{H4})$ (i), and
\[
h_{J+1}(t,x) := |x-\tilde{y}(t)|^2 - \delta_0^2.
\]
Hereafter, the proof of Theorem~\ref{pmo} follows the steps of Theorem~\ref{pm}, considering a modified setting in which we replace $h$ with $\tilde h$ and $C(t)$ with $C(t) \cap \left(\tilde y(t) + \delta_0 \mathbb{B} \right)$.
More precisely, we consider, for a given real number $\alpha >0$, the control system
\begin{equation*}
    (\widetilde{CS_\alpha})\; \left\{
    \begin{aligned}
        &\dot{x}(t) = f(t,x(t),u(t)) - \dfrac{1}{\alpha}\,\nabla_x  \tilde h (t,x(t)), \text{ a.e. } t \in [S,T], \\
        & u(t) \in U(t), \text{ a.e. } t \in [S,T], \\
        & |S-\bar S|+|T - \bar T|+ \|x_e -\bar x_e\|_{L^\infty([\bar{S}-\bar\delta,\bar{T}+\bar\delta],\R^n)}\leq \bar \delta.
    \end{aligned}
    \right.
\end{equation*}
We apply Lemma~\ref{lhy} below in place of Lemma~\ref{ingt}. This allows to derive a modified version of Lemmas~\ref{cvg} and \ref{stg} in which we replace hypothesis (H4) with $(\widetilde{H4})$ and the sweeping process (SP1) with 
\[
    ( \widetilde{SP1}) \quad
    \begin{cases}
    \dot{x}(t)=f(t,x(t),u(t)) - N_{C(t)\cap \left(\tilde y(t)+ \delta_0 \mathbb{B} \right)}(x(t)), \quad 
     \text{a.e. } t\in[S,T],\\
		  u(t) \in U(t), \text{ a.e. } t \in [S,T], \\
    x(S)=x_0 \in C(S).
    \end{cases}
    \]

\begin{lem}\label{lhy}
Assume that $(\widetilde{H4})$ holds. Then, there exist constants \(\gamma,\eta>0\) such that
\[
\bigl|\nabla_x \tilde h(t,x)\bigr|\ge \gamma\sum_{j=1}^{J+1} h_j^+(t,x)^2,
\]
for all $(t,x) \in \{ (s,y) \; : \; s\in [\bar{S}- \bar \delta,\bar{T}+ \bar \delta], \;  \tilde h (s,y) \le \eta \} $. 
\end{lem}

\begin{proof} 
We argue by contradiction. Suppose there exist $\eta_i \downarrow 0$, $\gamma_i \downarrow 0$,  $t_i \in [\bar S- \bar \delta, \bar T +\bar \delta]$ and $x_i$ such that $\tilde  h(t_i,x_i)\le \eta_i$ and 
\[
\bigl|\nabla_x \tilde h(t_i,x_i)\bigr|< \gamma_i\sum_{j=1}^{J+1} h_j^+(t_i,x_i)^2.
\]
Then, we would have
\[
\left| \sum_{j=1}^{J+1} \frac{h_j^+(t_i,x_i)^2}{\sum_{j=1}^{J+1} h_j^+(t_i,x_i)^2} \nabla_x h_j(t_i,x_i) \right| \to 0 \text{ as } i \to +\infty.
\]
Using a standard compactness argument, we can extract subsequences (we do not relabel) such that $(t_i,x_i) \to (t,x)$ and 
\[
\frac{h_j^+(t_i,x_i)^2}{\sum_{j=1}^{J+1} h_j^+(t_i,x_i)^2} \to \alpha_j, 
\]
for $j=1,\dots,J+1$, as $i \to +\infty$ for some $t \in [\bar S- \bar \delta, \bar T +\bar \delta]$, $x \in C(t)\cap (\tilde y (t) +  \delta_0 \B)$ and $\alpha_j \in [0,1]$ such that $ \sum_{j=1}^{J+1} \alpha_j =1$.
Passing to the limit yields
\[
\sum_{j \in I(t,x)} \alpha_j \nabla_x h_j(t,x) = 2 \alpha_{J+1}(\tilde y(t) - x),
\]
thus $\tilde y(t) - x \in N_{C(t)}(x)$. 
We distinguish two situations. 

\noindent
{\it Case 1:} $\sum_{j \in I(t,x)} \alpha_j \nabla_x h_j(t,x) =0$. But, then, from $(\widetilde{H4})$ (iii) we deduce that $\alpha_{1}= \dots=\alpha_{J}=0$, and, so, $\alpha_{J+1} =1 $ and $\tilde y(t) = x$. In view of the definition of $\alpha_{J+1}$ we also know that $h_{J+1}^+(t,x)=0$ and, then, $h_{J+1}(t,x)= 0$, that is $|\tilde y(t) - x|=\delta_0$. A contradiction.

\noindent
{\it Case 2:} $\sum_{j \in I(t,x)} \alpha_j \nabla_x h_j(t,x) \neq 0$. 
Then,  $\alpha_{J+1}(\tilde y(t) - x)\neq 0$ and $\alpha_{J+1}(\tilde y(t) - x)\in N_{C(t)}(\tilde y(t))$. Thus, $x \in \pi_{C(t)}(\tilde y(t))$ and $d_{C(t)}(\tilde y(t))=\delta_0$. But, $|\tilde y(t) - \bar x(t)| \leq \bar \delta_1$, that implies $d_{C(t)}(\tilde y(t))\le\delta_1< \delta_0$. A contradiction.
\end{proof}

\ \\

\noindent
According to the modified setting, we set
\[
\widetilde X :=\left\{ \left( [S,T],y,u \right) \text{ process for } ( \widetilde{SP1}) \;\; \text{ s.t. } 
|S-\bar S|+|T-\bar T|+|y(S)-\bar x(\bar S)| \leq \bar \delta/2 \right\},
\]
and, for each $i \in \N$, we consider the modified optimization problems $(\widetilde{SP_{i}})$, $(\widetilde{Q_{i}})$, and $(\widetilde{P_{i}})$, in which $\tilde h$ and $C(t) \cap \left(\tilde y(t)+ \delta_0 \mathbb{B} \right)$ respectively replace $h$ and $C(t)$ in $(SP_{i})$, $(Q_{i})$ and $(P_{i})$. Moreover, we obtain a process $([S_i,T_i],x_i,u_i)$ that is a local minimizer for the dynamic optimization problem $(\widetilde{P_{i}})$. 
Now, we observe that, for $i$ large enough, the constraint 
$\tilde y(t)+  \delta_0 \mathbb{B}$ is no longer active. 
Therefore, deriving the necessary optimality conditions for the local minimizer $([S_i,T_i],x_i,u_i)$, for each $i$ sufficiently large, $\tilde h$ coincides with $h$, and the reasoning reduces to the proof of Theorem~\ref{pm}.
 \qed

\section{Proof of the Free End-Time Maximum Principle (Lipschitz Time-dependent Case)}

As in Section \ref{Ek-sec}, we follow a similar line of reasoning, but here we emphasize the key differences that arise in the analysis.
\ \\

\noindent
{\bf Proof of Theorem \ref{pml}.} 
Take a sequence $(\varepsilon_i)_{i \in \N}$ converging to $0$ as in the previous section. 
We consider the following auxiliary optimization problem where the time interval $[\bar{S},\bar{T}]$ is, now, fixed:
\begin{equation*}
(\widehat{SP}_{i}): \left\{ \begin{array}{ll}
\text{Minimize } &
\big(g(\tau(\bar S),y(\bar S),\tau(\bar T),y(\bar T)) - g(\Bar{S},\Bar{x}(\Bar{S}),\Bar{T},\Bar{x}(\Bar{T})) + \varepsilon_i^2 \big) 
\\
& \hspace{3 in}
\vee d_E(\tau(\bar S),y(\bar S),\tau(\bar T),y(\bar T)) \\[2mm]
\text{over} & 
\text{arcs } (\tau,y) \in W^{1,1}([\bar S,\bar T];\R\times\mathbb{R}^{n}), \\
& \text{and measurable functions } (w,v) : [\bar S,\bar T] \to \R\times\R^{m}, \\
\text{such that} & \dot{\tau}(s) = w(s) \text{  for a.e. } s \in [\bar S,\bar T], \\
&\dot{y}(s) = w(s)f(\tau(s),y(s),v(s)) - N_{C(\tau(s))}( y(s)) \text{  for a.e. } s \in [\bar S,\bar T], \\
& w(s) \in [1-\beta_0,1+\beta_0] \text{  for a.e. } s \in [\bar S,\bar T],\\
& v(s) \in U \text{  for a.e. } s \in [\bar S,\bar T],\\
& h_j(\tau(\bar S),\bar y(\bar S))\le 0 \text{  for all } j=1,\dots,J ,\\
& |\tau(\bar S)-\Bar{S}|\le \delta_0, \;\;\; \|y-\Bar{x}\|_{L^\infty(\bar{S},\bar{T})} \leq \delta_0. \\
\end{array}
\right.
\end{equation*}
It is easy to see that for $\beta_0 \in (0, 1/2) $ and $\delta_0\in (0, \bar \delta/2)$ small enough, the process $((\bar{\tau}(s)=s,\bar{x}),(\bar{u},\bar{w}=1))$ is an $\varepsilon_i^2$-minimizer for $(\widehat{SP}_{i})$, and that the cost of $(\widehat{SP}_{i})$ is strictly positive. 
Applying the perturbation arguments of Section \ref{Ek-sec} (and adapted versions of Lemmas~\ref{cvg} and~\ref{stg}) we obtain a process $((\tau_i , y_i), (v_i, w_i) )$ which is a local minimizer for the following problem

\begin{equation*}
(\widehat{Q}_{i}): \left\{ \begin{array}{ll}
\text{Minimize } &
\left(g(\tau(\bar S),y(\bar S),\tau(\bar T),y(\bar T)) - g(\Bar{S},\Bar{x}(\Bar{S}),\Bar{T},\Bar{x}(\Bar{T})) +\varepsilon_i^2 \right) \vee d_E(\tau(\bar S),y(\bar S),\tau(\bar T),y(\bar T)) \\
&+ \varepsilon_i \left(|\tau (\bar S) - \tilde{\tau}_i (\bar S)|+|y(\bar S) - \tilde{y}_i(\bar S)|+\tilde{z}(\bar T) \right) \\
&+ \rho_i \left(|\tau (\bar S) - \tau_i (\bar S)|+|y(\bar S) - y_i(\bar S)|+z(\bar T) \right) \\
\text{over} & 
\text{arcs } (\tau,y,z) \in W^{1,1}([\bar S,\bar T];\R\times\mathbb{R}^{n}\times \R), \\
& \text{and measurable functions } (w,v) : [\bar S,\bar T] \to \R\times\R^{m}, \\
\text{such that} 
& \dot{\tau}(s) = w(s) \text{  for a.e. } s \in [\bar S,\bar T], \\
&\dot{y}(s) = w(s)f(\tau(s),y(s),v(s)) -\frac{w(s)}{\alpha_i} \nabla_x h(\tau(s), y(s)) \text{  for a.e. } s \in [\bar S,\bar T], \\
& \dot{\tilde z}(s)=\tilde{m}_i(s,v(s),w(s)) \text{  for a.e. } s \in [\bar S,\bar T], \\
& \dot{z}(s)=m_i(s,v(s),w(s)) \text{  for a.e. } s \in [\bar S,\bar T], \\
& w(s) \in [1-\beta_0,1+{\beta_0}] \text{  for a.e. } s \in [\bar S,\bar T],\\
& v(s) \in U \text{  for a.e. } s \in [\bar S,\bar T],\\
& h_j(\tau(\bar S),\bar y(\bar S))\le 0 \text{  for all } j=1,\dots,J ,\\
& \tilde{z}(\bar S)=0, \;\; z(\bar S)=0,
\end{array}
\right.
\end{equation*}
where
\(
\tilde{m}_i(s,v,w):=
\begin{cases}
\,1 & \text{if } (v,w)\neq(\tilde{v}_i(s),\tilde{w}_i(s)),\\
\,0 & \text{otherwise}.
\end{cases}
\)
and
\(
m_i(s,v,w):=
\begin{cases}
\,1 & \text{if } (v,w)\neq(v_i(s),w_i(s)),\\
\,0 & \text{otherwise}.
\end{cases}
\)
Here, $\rho_i \in (0, \varepsilon_i/2)$ and $\alpha_i > 0$ are suitably chosen numbers along the lines of the analysis of Section \ref{Ek-sec}, and $((\tilde \tau_i , \tilde y _i), ( \tilde v _i, \tilde w _i) )$ is a process such that
\begin{equation}\label{est3}
|\tilde \tau _i(\bar S) - \bar \tau (\bar S)|+|\tilde y _i(\bar S) - \bar x(\bar S)|+ d_{\mathcal{E}}((\tilde v _i,\tilde w _i),(\bar u, \bar w)) \le \varepsilon_i .
\end{equation}
Moreover, we have 
\begin{equation}\label{est4}
|\tau _i(\bar S) - \tilde \tau _i (\bar S)|+|y _i(\bar S) - \tilde y _i(\bar S)|+ d_{\mathcal{E}}(( v _i, w _i),(\tilde v _i,\tilde w _i)) \le \varepsilon_i .
\end{equation}

 \ \\

\noindent
{\bf Analysis of multipliers and their convergence properties.} 

\noindent
Applying the nonsmooth maximum principle, we obtain that there exist functions $(a_i,p_i) \in W^{1,1}([\bar S,\bar T];\R\times\mathbb{R}^{n})$, $\lambda_i^{(1)}, \ \lambda_i^{(2)}, \ \lambda_i^{(3)} \in [0,1]$ and $\beta^i_j\ge 0$, with $\lambda_i^{(1)} + \lambda_i^{(2)} + \lambda_i^{(3)} =1$ and $\sum_{j=1}^J \beta^i_j=1$, that satisfy  the following necessary conditions:
\begin{align}\label{lidif}
(\dot a_i(s), -\dot p_i(s)) & \in 
w_i(s)\operatorname{co}\partial_{t,x}\mathcal{H}(\tau_i(s),y_i(s),p_i(s),v_i(s))\\
& - w_i(s)\Big( \frac{1}{\alpha_i}p_i(s)\cdot \nabla_t \nabla_x h (\tau_i(s),y_i(s)),  
\frac{1}{\alpha_i}p_i(s) D_x^2 h(\tau_i(s),y_i(s))\Big) , \; a.e.\; s \in[\bar S,\bar T], \notag
\end{align}

\begin{equation}\label{litra}
\begin{aligned}
(-a_i(\bar S),p_i(\bar S),a_i(\bar T),-p_i(\bar T))
\in\;&
\lambda_i^{(1)}\,\partial g(\tau_i(\bar S),y_i(\bar S),\tau_i(\bar T),y_i(\bar T)) \\
&+ \lambda_i^{(2)}\,\partial d_E(\tau_i(\bar S),y_i(\bar S),\tau_i(\bar T),y_i(\bar T)) \\
&+ \lambda_i^{(3)}\,\sum_{j=1}^J \beta^i_j\nabla h_j(\tau_i(\bar S),y_i(\bar S))\times\{0\}\times\{0\}\\
&+ 2\varepsilon_i \,\big([-1,1]\times\B\times\{0\}\times\{0\}\big)
\end{aligned}
\end{equation}
and 
\begin{equation}\label{limax}
\begin{aligned}
w_i(s)\Big(&p_i(s)\cdot \big( f(\tau_i(s),y_i(s),v_i(s)) 
- \tfrac{1}{\alpha_i}\nabla_x h(\tau_i(s),y_i(s)) \big)
- a_i(s)\Big) \\
&\hspace{0.5 in}\ge w\Big(p_i(s)\cdot \big( f(\tau_i(s),y_i(s),v) 
- \tfrac{1}{\alpha_i}\nabla_x h(\tau_i(s),y_i(s)) \big)
- a_i(s)\Big)- 2\varepsilon_i,\\
&\hspace{2 in}\text{ for all } v\in U, \;  w \in [1-\beta_0,1+\beta_0], \;\; a.e.  s \in[\bar S,\bar T],
\end{aligned}
\end{equation}
For each $i\ge 1$, we write $D_i:=\{ t\in  [\bar S,\bar T] \; : \; (v_i(t),w_i(t)) =(\bar u (t), \bar w \equiv 1)\}$. In view of \eqref{est3}-\eqref{est4}, we know that $meas (D_i)\ge \bar T - \bar S - 2 \varepsilon_i$. 
From \eqref{limax} we deduce that 
\begin{equation}\label{liega}
    \left| p_i(s) \cdot \left( f(\tau_i(s),y_i(s),\bar u(s)) -\dfrac{1}{\alpha_i}\cdot \nabla_x h(\tau_i(s),y_i(s)) \right) -a_i(s) \right| \leq \frac{2\varepsilon_i}{\beta_0}, \;\; a.e. \; s \in D_i.
\end{equation}
Similarly as in Section \ref{Ek-sec}, we obtain that $(p_i)_{i \in \N}$ is (uniformly) bounded in $W^{1,1}([\bar S,\bar T];\mathbb{R}^{n})$ and that $\left( \frac{1}{\alpha_i} \overset{J}{\underset{j=1}{\sum}} h^+_j(\tau_i,y_i) |p_i\cdot \nabla_x h_j(\tau_i,y_i)| \right)_{i \in \N}$ is (uniformly) bounded in $L^{1}([\bar S,\bar T];\mathbb{R}^{n})$. 
Now, we examine the behaviour of $(a_i)_{i \in \N}$. Observe that for a.e. $s \in[\bar S,\bar T]$ we have
\begin{align*}
\frac{1}{\alpha_i}p_i(s)\cdot \nabla_t \nabla_x h (\tau_i(s),y_i(s))= 
&-\frac{1}{\alpha_i}\sum_{j=1}^J h_j^+(\tau_i(s),y_i(s))^2 \, p_i(s)\,\nabla_t\nabla_x h_j(\tau_i(s),y_i(s)) \\
&-\frac{2}{\alpha_i}\,\sum_{j=1}^J h_j^+(\tau_i(s),y_i(s))\, 
\nabla_t h_j(\tau_i(s),y_i(s))\, p_i(s)\cdot \nabla_x h_j(\tau_i(s),y_i(s)) .
\end{align*}
Reasoning as in Section \ref{Ek-sec}, we obtain that the sequence
$$
\frac{1}{\alpha_i}\,\sum_{j=1}^J h_j^+(\tau_i(.),y_i(.))\, 
|\nabla_t h_j(\tau_i(.),y_i(.))|\, |p_i(.)\cdot \nabla_x h_j(\tau_i(.),y_i(.))|
$$
is uniformly bounded in $L^1$. 

Write $\widetilde K:=\underset{j \in \{1,\dots, J\}}{\underset{(s,x) \in [\bar S - \bar \delta, \bar T + \bar \delta]\times R\B}{\max}} |
\nabla_t\nabla_x h_j(s,x)|$. 
The first inclusion of \eqref{lidif} gives, for a.e.\ $s \in[\bar S,\bar T]$,
\[
|\dot  a_i| \leq
(1+\beta_0)
\left(k_f + 
(3A)^{2/3}k_2\widetilde K + 
\dfrac{2K_h}{\alpha_i}\overset{J}{\underset{j=1}{\sum}} h^+_j(\tau_i,y_i)  p_i\cdot \nabla_x h_j(\tau_i,y_i) \right).
\]
Therefore $(a_i)_{i \in \N}$ is also (uniformly) bounded in 
$W^{1,1}([\bar S,\bar T];\mathbb{R})$.

Applying Helly's selection theorem, we may extract some subsequences of $(a_i)_{i\in\N}$ and $(p_i)_{i\in\N}$ (we do not relabel) that converge pointwise on $[\bar S,\bar T]$ respectively to some $BV$ functions $a$ and $p$, and such that the associated measures $(\dot a_i(t)\,\dd t)$ and $(\dot p_i(t)\,\dd t)$ converge weakly$^*$ in $C^0([\bar S,\bar T];\R^n)^*$ respectively to the (signed) Radon measures $\dd a$ and $\dd p$.

Along the same subsequence we may also assume that
$\lambda_i^1\to\lambda\in[0,1]$, 
and, by the Banach--Alaoglu theorem, we may extract a further subsequence (we do not relabel) such that for each $j \in \{1,\dots,J\}$, the sequence of functions $\displaystyle \left(\frac{h_j^+\!\bigl(\tau_i,y_i\bigr)^2}{\alpha_i}\right)_{i \in \N}$ converges weakly in $L^2(\bar S,\bar T)$ to a function $\eta_j$ that takes values in $[0,(3A)^{2/3}k_2]$.
This implies condition (ii) of Theorem \ref{pml}.

Moreover, for each $j \in \{1,\dots,J\}$, from the sequence of measures
\[
\left(\frac{1}{\alpha_i} h^+_j(\tau_i(t),y_i(t)) p_i(t)\cdot \nabla_x h_j(\tau_i(t),y_i(t)) \, \dd t \right)_{i \in \N}
\]
we can extract a subsequence (we do not relabel) that converges weakly$^*$ in $C^0([\bar S,\bar T];\R^n)$ to Radon measure $\dd \mu_j$.  
Combined with the definition of $(\eta_j)_{1\leq j\leq J}$, this yields conditions~(iii) and (ix) of Theorem \ref{pml}. Since, for each $j \in \{1,\dots,J\}$,
\[
\left(
\frac{1}{\alpha_i}
h_j^+\!\bigl(\tau_i(t),y_i(t)\bigr)\,
\bigl|p_i(t)\cdot\nabla_x h_j(\tau_i(t),y_i(t))\bigr|^2
\,\dd t
\right)_{i\in\N}
\]
is a sequence of positive measures, passing to the limit gives condition (viii) of Theorem \ref{pml}. Also, for each $j \in \{1,\dots,J\}$
\[
\eta_j(t)\,p(t)\cdot\nabla_x h_j(t,\bar x(t)) \dd t= h_j^+ (t,\bar x(t)) \dd \mu_j (t) = 0.
\]
from which condition (vii) of Theorem \ref{pml} follows. We recognize \eqref{litra}, \eqref{liega} and \eqref{limax} as perturbed versions of conditions (iv), (v) and (vi) of Theorem \ref{pml}.  In the limit-taking procedure, the analysis is similar to the one employed in the proof of Theorem \ref{pm}; therefore, the details are omitted.
\hfill $\square$

\ \\

\noindent
{\bf Proof of Proposition \ref{libre1}.} 
The proof is along similar lines to that of Proposition \ref{libre} and is therefore omitted. 
\qed

\ \\

\noindent
{\bf Proof of Theorem \ref{pmo1}.} 
The argument follows the same reasoning as in the proof of Theorem \ref{pmo} and the proof Theorem \ref{pml}, and is therefore omitted. 
\qed

\ \\

\noindent
{\bf Acknowledgments.} This research was partially funded by the French Agence Nationale de la Recherche (ANR), under grant ANR-22-CE40-0010-01 (project COSS), by ANR-FWF International (French-Austrian) Grant FWF 4368255 and ANR-25-CE40-3469-01 (project SABOCPR). This project has received financial support from the CNRS through the MITI interdisciplinary programs.

\ \\

\end{document}